\documentclass[a4paper, 10pt]{amsart}
\usepackage[top=80pt,bottom=65pt,left=70pt,right=70pt]{geometry}
\usepackage{graphicx, eepic}
\usepackage{times}
\normalfont
\usepackage{bbm}
\usepackage{xcolor}
\usepackage{kotex}
\usepackage{verbatim}
\usepackage{mathrsfs}
\usepackage{amscd}
\usepackage{float}
\usepackage[all]{xy}
\usepackage[subnum]{cases}
\ifpdf
\usepackage{tikz}
\fi
\usetikzlibrary{arrows,matrix,positioning}

\usepackage{amsthm,amsmath,amsfonts,amssymb}
\usepackage[colorlinks=true, linkcolor=blue, citecolor=red, filecolor=violet, urlcolor=violet]{hyperref}
\usepackage{multirow, url}
\usepackage{color}
\usepackage[all]{xy}
\theoremstyle{plain}
\newtheorem{theorem}{Theorem}[section]
\newtheorem{thmx}{Theorem}

\newtheorem{proposition}[theorem]{Proposition}
\newtheorem{lemma}[theorem]{Lemma}
\newtheorem{corollary}[theorem]{Corollary}

\theoremstyle{definition}

\newtheorem{definition}[theorem]{Definition}
\newtheorem{example}[theorem]{Example}

\theoremstyle{remark}
\newtheorem{remark}[theorem]{Remark}

\newcommand{\C}{\mathbb{C}}

\newcommand{\N}{\mathbb{N}}

\newcommand{\R}{\mathbb{R}}

\newcommand{\Z}{\mathbb{Z}}

\newcommand{\fa}{\frak{a}}
\newcommand{\fb}{\frak{b}}

\newcommand{\rmso}{\mathrm{SO}}
\newcommand{\rmu}{\mathrm{U}}

\def\mcal{\mathcal}
\def\frak{\mathfrak}
\def\scr{\mathscr}

\numberwithin{equation}{section} \numberwithin{table}{section}

\begin{document}                                                                          

\title{Lagrangian fibers of Gelfand--Cetlin systems of $\mathrm{SO}(n)$-type}
\author{Yunhyung Cho}
\address{Department of Mathematics Education, Sungkyunkwan University, Seoul, Republic of Korea}
\email{yunhyung@skku.edu}

\author{Yoosik Kim}
\address{Department of Mathematics, Brandeis University, Waltham, MA, USA}
\email{kimyoosik27@gmail.com, yoosik@brandeis.edu}


\begin{abstract}
	In this paper, we study the Gelfand--Cetlin systems and polytopes of the co-adjoint $\mathrm{SO}(n)$-orbits.
	We describe the face structure of Gelfand--Cetlin polytopes and iterated bundle structure of Gelfand--Cetlin fibers in terms of combinatorics on the ladder diagrams. 
	Using this description, we classify all Lagrangian fibers.   
\end{abstract}
\maketitle
\setcounter{tocdepth}{1} 
\tableofcontents

\section{Introduction}
\label{secIntroduction}

Let $G$ be a compact connected Lie group. 
Guillemin and Sternberg \cite{GS_GC} observed that an integrable system of collective Hamiltonian functions constructed by Thimm's trick \cite{Thimm} on a Hamiltonian $G$-manifold $(M, \omega, \mu)$ generates a Hamiltonian torus action on an open dense subset of $M$.  
For instance, this construction produces integrable systems on co-adjoint $G$-orbits by applying the trick to a nested sequence of connected closed subgroups of $G$.
In particular, when $G$ is a unitary group or a special orthogonal group, these integrable systems are \emph{complete}. Guillemin and Sternberg called them \emph{Gelfand--Cetlin systems} (\emph{GC systems} for short) in connection with Gelfand--Cetlin bases for unitary and orthogonal irreducible representations \cite{GC}.
Moreover, the images of Gelfand--Cetlin systems were shown to be convex polytopes determined by inequalities from the so-called ``min-max principle", which are called the \emph{Gelfand--Cetlin polytope} (\emph{GC polytope} for short) \cite{GS_GC}. 

The GC systems and polytopes have been in the spotlight in several branches of mathematics. 
After the relation between the dimension of a weight subspace of the irreducible representation with the highest weight $\lambda$ and the number of integer lattice points of certain patterns in the polytope associated to $\lambda$ was discovered, the GC polytopes started to draw attention in representation theory, see \cite{GC, Zelo, Lit, DMc} for instance.
As toric degenerations of flag varieties and GC systems were constructed in algebraic and symplectic geometry, the GC systems and polytopes have been used to study Schubert calculus, mirror symmetry, and Floer theory of flag varieties in \cite{GL, BCKV, Cal, KM, NNU1} and many others. 
In the theory of Poisson--Lie groups, the GC systems on $\frak{u}(n)^*$ and the dual Poisson--Lie group $\mathrm{U}(n)^*$ have been applied to study Ginzburg--Weinstein diffeomorphisms (see \cite{AM, ALL, FR}).

The main purpose of this article is to understand the GC polytopes and systems in detail in the case where $G$ is a special orthogonal group. 
In the case of unitary groups, the first named author with B. An and J. Kim \cite{ACK} found the correspondence between the faces and certain subgraphs of the \emph{ladder diagram}. 
Using their description of GC polytopes, the authors with Y.-G. Oh  studied Hamiltonian (non-)displaceability of GC fibers of the co-adjoint $\rmu(n)$-orbits and proved that 
every GC fiber is isotropic and is the total space of a certain iterated bundle in \cite{CKO}. 
In particular, they provided a combinatorial dimension formula for the GC fibers and located all Lagrangian fibers over the GC polytopes. 

To extend the works of \cite{ACK, CKO} to the $\rmso(n)$-case, we introduce a certain partially ordered set of \emph{pairs} of subgraphs of a \emph{ladder diagram} of $\rmso(n)$-type and find an order-preserving one-to-one correspondence
between the set of faces of a GC polytope and the above pairs of the corresponding ladder diagram, see Theorem~\ref{thm_diacorrespo}. With the aid of the description, we then express a GC fiber as the total space of a certain 
iterated bundle.

\begin{thmx}[Theorem \ref{thm_fiber}]\label{thmx_main}
Let $\Phi_\lambda$ be the Gelfand--Cetlin system on the co-adjoint orbit $\mcal{O}_\lambda$ of $\lambda \in \frak{so}(n)^*$ equipped with a Kirillov--Kostant--Souriau symplectic form.  
For any point ${\bf{u}}$ in the Gelfand--Cetlin polytope $\Delta_\lambda$, the fiber $\Phi_\lambda^{-1}({\bf{u}})$ is an isotropic smooth submanifold of $\mcal{O}_\lambda$ and is the total space of an iterated bundle
\begin{equation}\label{equ_iteratedthma}
\Phi_\lambda^{-1}({\bf{u}}) = E_{n} \to E_{n-1} \to \cdots \to E_2  = \textup{point}
\end{equation}
such that the fiber at each stage is either a point or a product of spheres.
\end{thmx}

The process in Theorem~\ref{thmx_main}, interacting with the face structure of the GC polytope, provides a concrete and detailed description for the GC fiber in response to the location of a fiber in the GC systems on co-adjoint $\mathrm{SO}(n)$-orbits.  
The fiber at each stage will be shown to be a homogeneous space in Section~\ref{secIteratedBundleStructureOfGelfandCetlinFibers} 
which partly explains why the list of diffeomorphism types of fibers at stage in~\eqref{equ_iteratedthma} is short. 

Note that an \emph{even} dimensional sphere factor can appear in a stage of~\eqref{equ_iteratedthma}, while it cannot for $\rmu(n)$-cases (cf. \cite[Theorem A]{CKO}). 
The underlying combinatorics in the ladder diagram for the $\mathrm{SO}(n)$-type is indeed more interesting than the $\mathrm{U}(n)$-type. 
For instance, even if the patterns of inequalities for two points in the polytope are same, whether a component of the GC system becomes zero or not \emph{does} make a distinction on the topology of the GC fiber. 
Also, the combinatorial type of GC polytopes of the $\mathrm{SO}(n)$-type is affected by both the inequality patterns on $\lambda$ and the zeros in $\lambda$, but in the $\mathrm{U}(n)$-case two GC polytopes are combinatorially same as long as their inequality patterns are same.  

For the $\mathrm{U}(n)$-type, the authors with Oh \cite{CKO} and Bouloc--Miranda--Zung \cite{BMZ} independently studied the GC fibers from different motivations. According to \cite[Problem 1.9]{BMMT}, peculiar properties of GC systems was predicted by N. T. Zung. Along the line, for the systems of bending flow on the moduli space of polygons and generalized GC systems on the Grassmannian of complex two planes, Bouloc \cite{Bou} proved that the fibers are smooth and isotropic. 
In this regard, Theorem~\ref{thmx_main} confirms the peculiarity of the GC systems on the $\mathrm{SO}(n)$-orbits.
Lane \cite{Lane1, Lane2} proved that each GC fiber is a smooth manifold and has an iterated bundle structure in a more general situation including the co-adjoint $\mathrm{SO}(n)$-orbits. 
We remark that the iterated bundle description in \cite{Lane2} is slightly different from~\eqref{equ_iteratedthma}.

To maximize efficiency, we enhance the combinatorial process for Theorem~\ref{thmx_main} that enables us to count the dimension of the fiber instantly, which can be used to classify the faces containing Lagrangian fibers in their relative interiors. Roughly speaking, once a face of a GC polytope and the corresponding pair of subgraphs of a ladder diagram are given, 
one can immediately count the dimension of the fiber over the interior of the face by doing certain Tetris game of ``L-blocks'' and ``I-blocks'', see Section \ref{secClassificationOfLagrangianFaces}. 

\begin{thmx}[Corollary \ref{cor_fillable}]\label{thmx_main2}
	Let $f$ be a face of $\Delta_\lambda$ and $\gamma_f$ be the corresponding face of $\Gamma_\lambda$. Then the GC fiber over any point ${\bf u}$ in the relative interior of $f$ 
	is Lagrangian if and only if $\gamma_f$ is fillable by $L$-blocks and $I$-blocks.
\end{thmx}

The organization of the paper is in order. In Section~\ref{secGelfandCetlinSystems}, we briefly review to GC systems of the $\mathrm{SO}(n)$-type. In Section \ref{secTheFaceStructureOfGelfandCetlinPolytopes}, we describe the face structures of GC polytopes in terms of combinatorics of the corresponding ladder diagrams. 
After introducing underlying combinatorics in Section~\ref{secTheTopologyOfGelfandCetlinFibers},
we discuss the iterated bundle structure of the GC fibers and the diffeomorphism type of fibers at stages in Section~\ref{secIteratedBundleStructureOfGelfandCetlinFibers}
and~\ref{sectionproofoflemma}.
Finally, Theorem \ref{thmx_main2} and the dimension formula will be derived in Section~\ref{secClassificationOfLagrangianFaces}.

\subsection*{Acknowledgement} 
The authors express our deep gratitude to Yong-Geun Oh for drawing their attention to Gelfand--Cetlin systems. 
The authors would like to thank the referees for their careful, constructive, and detailed comments to improve the quality of the
manuscript greatly. 
This work was initiated when the second named author was affiliated to IBS-CGP and was supported by IBS-R003-D1. The first named author is supported by the National Research Foundation of Korea(NRF) grant funded by the Korea government(MSIP; Ministry of Science, ICT \& Future Planning) (NRF-2017R1C1B5018168).

\section{Gelfand--Cetlin systems}
\label{secGelfandCetlinSystems}

Let $G$ be a compact connected Lie group. 
The Lie algebra of $G$ and its dual are denoted by $\frak{g}$ and $\frak{g}^*$ respectively.
A $G$-action on symplectic manifold $(M,\omega)$ is called {\em Hamiltonian} if
\begin{enumerate}
\item
 there is a smooth map $\mu 
\colon M \rightarrow \frak{g}^*$ such that 
\[
	\omega(\underline{X}, \cdot) = d\mu_X(\cdot) \quad \quad \left(\mu_X := \langle \mu, X \rangle \colon M \rightarrow \R\right)
\]
for each $X \in \frak{g}$ where $\underline{X}$ is the fundamental vector field on $M$ generated by $X$. 
\item $\mu$ is required to be a $G$-equivariant map with respect to the co-adjoint $G$-action on $\frak{g}^*$
\end{enumerate}
Such a function $\mu$ is called a \emph{moment map} (for the $G$-action) and $(M, \omega, \mu)$ is called a \emph{Hamiltonian $G$-manifold}. 

In this section, we review the construction of integrable systems on a Hamiltonian $G$-manifold in \cite{Thimm, GS_GC}, focusing on the \emph{(classical) Gelfand--Cetlin system} on a co-adjoint $\mathrm{SO}(n)$-orbit.

\subsection{Collective periodic Hamiltonians}
\label{ssecCollectivePeriodicHamiltonians}

Let $(M,\omega, \mu)$ be a Hamiltonian $G$-manifold. 
We denote by $C^\infty(M)$ the set of smooth functions on $M$. 
The symplectic structure $\omega$ induces the Poisson structure $\{ \cdot, \cdot \}_M$ defined by 
	\[
		\{ f, g \}_M := \omega(H_f, H_g), \quad \quad f,g \in C^\infty(M)
	\]
	where $H_f$ and $H_g$ denote the Hamiltonian vector fields generated by $f$ and $g$ respectively. On the other hand, the Lie algebra $\frak{g}^*$ also admits the Poisson structure $\{ \cdot, \cdot \}_{\frak{g}^*}$ given by 
	\[
		\{h, k\}_{\frak{g}^*}(\xi) := \langle \xi, [dh_\xi, dk_\xi] \rangle, \quad \quad h, k \in C^\infty(\frak{g}^*)
	\]
	where $dh_\xi \colon T_\xi \frak{g}^* \cong \frak{g}^* \rightarrow \R$ is regarded as an element of $\frak{g} = \mathrm{Hom}(\frak{g}^*, \R)$.
	The moment map $\mu$ is Poisson, i.e., 
	\begin{equation}\label{equation_Poisson}
		\{h, k\}_{\frak{g}^*}(\mu(x)) = \{h \circ \mu, k \circ \mu\}_M(x), \quad \quad x \in M
	\end{equation}
	for every $h, k \in C^\infty(\frak{g}^*)$, see \cite[Proposition III.1.3]{Au} for instance.
	Thus, $h \circ \mu$ and $k \circ \mu$ commute with respect to $\{ \cdot, \cdot \}_M$ if $h$ and $k$ commute with respect to $\{ \cdot, \cdot \}_{\frak{g}^*}$. 
	Recall the following identitiy that will be used in Section \ref{secIteratedBundleStructureOfGelfandCetlinFibers}.

	\begin{lemma}\label{lemma_preserve}
		For any $X, Y \in \frak{g}$ and $x \in M$, we have 
		\begin{equation}\label{equation_multiplicity_free_lemma}
			\omega_x(\underline{X}_x, \underline{Y}_x) = \omega_{\mu(x)}^{\mathrm{KKS}} (d\mu_x(\underline{X}_x), d\mu_x(\underline{Y}_x))
		\end{equation}
		where $\omega^{\mathrm{KKS}}_\xi$ denotes the Kirillov--Kostant--Souriau symplectic form on the co-adjoint orbit $\mcal{O}_\xi \subset \frak{g}^*$. That is, 
		$\omega^{\mathrm{KKS}}_\xi (\underline{X}, \underline{Y}) = \langle \xi, [X,Y] \rangle$ for any $X, Y \in \frak{g}^*$.
	\end{lemma}

	\begin{proof}
		Observe that (LHS) of \eqref{equation_multiplicity_free_lemma}
		equals $\{\mu_X, \mu_Y\}_M(x)$. On the other hand, (RHS) of \eqref{equation_multiplicity_free_lemma} is
		\[
			\begin{array}{ccl}
				\mathrm{(RHS)} & = & \omega_{\mu(x)}^{\mathrm{KKS}}( \underline{X}_{\mu(x)}, \underline{Y}_{\mu(x)}) \quad \quad (\text{since $\mu$ is an equivariant map}) \\
						& = & \langle \mu(x), [X,Y] \rangle \\ 
						& = & \langle \mu(x), [df, dg] \rangle \quad \quad (f = \langle \cdot , X \rangle, ~g = \langle \cdot, Y \rangle : \frak{g}^* \rightarrow \R)\\ 
						& = & \{f, g\}_{\frak{g}^*} (\mu(x)) \\ 
						& = & \{ f\circ \mu, g \circ \mu \}_M (x) \quad \quad (\text{by ~\eqref{equation_Poisson}}) \\
						& = & \{\mu_X, \mu_Y \}_M(x).
			\end{array}
		\]
		This finishes the proof.
	\end{proof}
	
We recall the definition of a (completely) integrable system on $(M, \omega)$. 

\begin{definition}\label{definition_cis}
An {\em integrable system} on an $2n$-dimensional symplectic manifold $(M,\omega)$ is a continuous function $(f_1, \cdots, f_k) \colon M \rightarrow \R^k$ $(k \leq n)$ such that 
\begin{itemize}
	\item each $f_i$ is smooth, 
	\item $\{f_i, f_j \}_M = 0$ for every pair $(i,j)$, 
	\item $\{df_1, \cdots, df_k\}$ are linearly independent 
\end{itemize}
on some open dense subset $\mcal{U}$ of $M$. 
An integrable system $(f_1, \cdots, f_k)$ is called {\em completely integrable} if $k = n$. 
\end{definition}

	 A {\em collective Hamiltonian function} is defined as a function on $M$ of the form $h \circ \mu$ for some real-valued smooth function $h$ on $\frak{g}^*$.
Let $C^\infty(\frak{g}^*)^G$ be the set of $G$-invariant smooth functions (under the co-adjoint $G$-action). 
Any $G$-invariant function $h \in C^\infty(\frak{g}^*)^G$ Poisson commutes with any functions in $C^\infty(\frak{g}^*)$.
If $h_1, \cdots, h_k \in C^\infty(\frak{g}^*)^G$, then the corresponding collective Hamiltonian functions $h_1 \circ \mu, \dots, h_k \circ \mu$ are commutative with respect to $\{ \cdot, \cdot \}_M$. 

In order to construct $G$-invariant functions on $\frak{g}^*$, let us fix a Cartan subgroup $T$ of $G$. 
Let $\frak{t}$ and $\frak{t}^*$ be the Lie algebra of $T$ and its dual, respectively. 
Also, we choose a positive Weyl chamber $\frak{t}^*_+ \subset \frak{t}^*_{\vphantom{+}}$. 
Every co-adjoint $G$-orbit intersects 
$\frak{t}_+^*$\footnote{Note that $\frak{t}_+^* \cong \frak{t} / W   \cong \frak{g}^* / \text{Ad}^*(G)$
where $\text{Ad}^*$ denotes the co-adjoint action of $G$ on $\frak{g}^*$ and $W$ denotes the Weyl group of $G$.}
at a single point, see \cite[Lemma 3]{Ki}. Thus, the following map is well-defined$\colon$
\begin{equation}\label{equ_pipi}
	\begin{array}{cccl}
		\Pi \colon & \frak{g}^* & \rightarrow & \frak{t}_+^* \\
			  &   \lambda   &  \mapsto   &\Pi(\lambda) := \left(G\cdot \lambda\right) \cap \frak{t}_+^*.
	\end{array}
\end{equation} 
The map $\Pi$ is surjective, continuous on $\frak{g}^*$, and smooth on $\Pi^{-1}\left(\text{int}(\frak{t}^*_+)\right)$. 
For any $\xi \in \frak{t}$, 
let
\begin{equation}\label{equ_linearfunctionall}
	l_\xi \colon \frak{t}^* \rightarrow \R, \quad l_\xi(\lambda) := \langle \lambda, \xi \rangle.
\end{equation}
The composition $h_\xi := \l_\xi \circ \Pi \colon \frak{g}^* \rightarrow \R$ is then a $G$-invariant smooth function on $\Pi^{-1}\left(\text{int}(\frak{t}^*_+)\right)$, 
which leads to the collective Hamiltonian function $h_\xi \circ \mu$.
In general, the set of those independent collective Hamiltonian functions is \emph{not} enough to form a \emph{completely} integrable system on $(M,\omega)$.

\subsection{Thimm's trick}
\label{ssecRmsoNCases}

To obtain additional independent commuting functions on $\mathfrak{g}^*$, we apply Thimm's trick \cite{Thimm} as follows.
Consider a nested sequence of connected closed Lie subgroups of $G$$\colon$
\begin{equation}\label{equation_nested_sequence}
	G_\bullet := \left\{\langle e \rangle =: G_0 \subset G_{1} \subset \cdots \subset G_{k-1} \subset G_k := G \right\}.
\end{equation}
For each $i=1,\cdots, k$, let $\frak{g}^{\vphantom{*}}_i$ and $\frak{g}_i^*$ be the Lie algebra of $G_i$ and its dual, respectively. 
Choose a Cartan subgroup $T_i$ of $G_i$ and denote by $\frak{t}^{\vphantom{*}}_i$, $\frak{t}_i^*$, and $\frak{t}_{i,+}^*$ the 
Lie algebra of $T_i$, its dual, and a positive Weyl chamber, respectively. 

The Hamiltonian $G$-action on $M$ induces a Hamiltonian $G_i$-action  on $M$ where the composition
\[
	\Phi_i := \kappa_i \circ \mu \colon M \rightarrow \frak{g}_i^* 
\]
is a moment map of the induced $G_i$-action. Here, $\kappa_i \colon \frak{g}^* \rightarrow \frak{g}_i^*$ 
is the transpose of the linear embedding $\frak{g}_i \hookrightarrow \frak{g}$.
As in~\eqref{equ_pipi}, we can also define the map 
$$
\Pi_i \colon \frak{g}^*_i \to \frak{t}_{i,+}^* \quad \quad i=1,\cdots,k.
$$ 
Fix a basis $\{\xi_{ij} ~|~ j=1,\cdots,r_i \}$ (where $r_i$ is the rank of $T_i$)
of the lattice (i.e., the kernel of $\exp \colon \frak{t}_i \rightarrow T_i$) in $\frak{t}_i$, which gives rise to the linear functionals $\{l_{\xi_{ij}} ~|~ j=1,\cdots,r_i \}$ in~\eqref{equ_linearfunctionall}.
Then the composition
\begin{equation}\label{equ_compoGCsys}
		\Phi^{(i)}_j := l_{\xi_{ij}} \circ \Pi_i \circ \Phi_i = l_{\xi_{ij}} \circ \Pi_i \circ \kappa_i \circ \mu \colon M \to \frak{g}^*  \to  \frak{g}_i^*  \to  \frak{t}_{i,+}^*  \to  \R
\end{equation}
becomes a periodic Hamiltonian, i.e., each function $\Phi^{(i)}_j$ generates a Hamiltonian $S^1$-action, on the open dense subset $\mcal{U}_j^{(i)}$ on which $\Phi^{(i)}_j$ is smooth, see \cite[Theorem 3.4]{GS_GC}.

	\begin{definition}\label{definition_GC_Multiplicity_free}
		Let $(M,\omega, \mu \colon M \rightarrow \frak{g}^*)$ be a Hamiltonian $G$-manifold. 
		Let $\Psi$ be the set of non-constant collective Hamiltonian functions of the form $\Phi_j^{(i)} \circ \mu$ where $\Phi_j^{(i)}$ is in~\eqref{equ_compoGCsys}. 
		If $\Psi$ happens to be a completely integrable system, then we call $\Psi$ a \emph{Gelfand--Cetlin system on} $M$. 
	\end{definition}
			
	The following theorem was proved by Guillemin and Sternberg under the assumption that $\mu(M)$ is a submanifold of $\frak{g}^*$. Recently Lane proved this in a general setting.
	
	\begin{theorem}[p. 224 in \cite{GS_Th}, Proposition 8 in \cite{Lane1}]\label{theorem_multiplicityfree}
		If a Hamiltonian $G$-manifold $(M,\omega,\mu)$ admits a Gelfand--Cetlin system, then the $G$-action is multiplicity free\footnote{Recall that a Hamiltonian $G$-action
		is {\em multiplicity free} if the symplectic reduction at every value in the image of the moment map is a point. Equivalently, the $G$-action is multiplicity free 
		if the induced $G$-action on $\mu^{-1}(\mcal{O}_\xi)$ is transitive for every $\xi \in \frak{g}^*$ where $\mcal{O}_\xi$ denotes the co-adjoint $G$-orbit 
		of $\xi$. See \cite{GS_Th} for more details.}.
	\end{theorem}

\subsection{Co-adjoint $\rmso(n)$-orbits}
\label{ssec_soNCases}

Let $G$ be a real special orthogonal group $\rmso(n)$ where $n = 2\nu+1$. Take the following maximal torus $T$$\colon$ 
\[
	T := \left\{ \begin{pmatrix}
		R(\theta_1) & \cdots & 0 & 0\\
		\vdots & \ddots & \vdots & \vdots\\
		0  & \cdots & R(\theta_\nu) & 0 \\
		0         & \cdots & 0 & 1
	\end{pmatrix} \in \rmso(n) ~\colon~ (\theta_1, \cdots, \theta_{\nu}) \in \R^{\nu} \right\} 
	\,\, \mbox{where }
	R(\theta) = 
	\begin{pmatrix} \cos \theta & -\sin \theta \\ \sin \theta & \cos \theta \end{pmatrix}.
\]
The Lie algebra $\frak{g}$ is the set $\mcal{S}_{n}$ of real skew-symmetric $(n \times n)$ matrices. 
Equip $\frak{g}$ with the inner product given by
\begin{equation}\label{equation_metric}
	\langle X, Y \rangle := - \frac{1}{2} \mathrm{Tr}(XY), \quad \mbox{for $X,Y \in \mcal{S}_{n}$},
\end{equation}  
which defines an $\mathrm{Ad}(G)$-invariant metric. Take an orthonormal basis $\{ e_{ij} \}_{1 \leq i < j \leq 2\nu+1}$ for $\frak{g}$ with respect to the bilinear form~\eqref{equation_metric} where $e_{ij} \in \mcal{S}_n$ is defined by 
\begin{equation}\label{equ_choiceoforthonomal}
	(e_{ij})_{pq} = \begin{cases} 
		 {1}  \quad &\text{if}~ (i,j) = (p,q) \\
		 -{1} \quad &\text{if} ~(i,j) = (q,p) \\ 
		 0 \quad & \text{otherwise.}
		\end{cases}
\end{equation}

For $i=1,\cdots, \nu$, we particularly set $h_i := e_{2i-1,2i}$. Note that the form $\langle \cdot, \cdot \rangle$ induces a natural identification of $\frak{g}$ with $\frak{g}^*$. In terms of the orthonormal basis $\{e_{ij}\}$ in~\eqref{equ_choiceoforthonomal}, the identification can be described as its dual basis$\colon$ 
\begin{equation}\label{equ_dualformidentification}
	\begin{array}{ccl}
		\frak{g}^{\vphantom{*}} & \rightarrow & \frak{g}^* \\
			e^{\vphantom{*}}_{ij} & \mapsto & e_{ij}^*
	\end{array}
	\quad \quad e_{ij}^*(e^{\vphantom{*}}_{pq}) = \begin{cases} 1 \hspace{0.5cm} \text{if}~(i,j) = (p,q) \\ 0 \hspace{0.5cm} \text{otherwise.} \end{cases}
\end{equation}
Also, the Lie algebra $\frak{t} \subset \frak{g}$ of $T$ consists of elements of the form
\[
\frak{t} = \left\{ \, \begin{pmatrix}
		B({\lambda_1}) & \cdots & 0 & 0\\
		\vdots & \ddots & \vdots & \vdots  \\
		0  & \cdots & B({\lambda_\nu}) & 0 \\
		0         & \cdots & 0 & 0
	\end{pmatrix} \in \mcal{S}_{2\nu+1} ~:~ (\lambda_1, \cdots, \lambda_\nu) \in \R^\nu
	\right\}
\,\, \mbox{where }
B(\lambda) := \begin{pmatrix} 0 & \lambda \\ - \lambda & 0 \end{pmatrix}.
\]
Moreover, the Lie algebra $\frak{t}$ can be identified with $\R^\nu$ by the map
\begin{equation}\label{equ_dualformidentification2}
\sum_{i=1}^\nu  \lambda_i \, h_i \in \frak{t} \mapsto (\lambda_1,\cdots,\lambda_\nu) \in \R^\nu.
\end{equation}

Under the identifications~\eqref{equ_dualformidentification} and~\eqref{equ_dualformidentification2}, the linear functional $h_i^*$ on $\frak{t}$ becomes the linear functional on $\R^\nu$ defined by $h_i^*(\lambda_1,\cdots, \lambda_\nu) = \lambda_i$.
The root of $\frak{g}$ consists of $\{\pm h_i^*\}_{i=1,\cdots,\nu}  \cup \{ \pm(h_i^* \pm h_j^*) \}_{i \neq j}$. A base can be chosen as
\[
	\Delta = \left\{h_1^* - h_2^*, \cdots, h_{\nu-1}^* - h_\nu^*, h_\nu^* \right\} \subset \frak{t}^*.
\]
Then the corresponding positive Weyl chamber is given by
\begin{equation}\label{equ_positiveweylodd}
	\frak{t}_+^* = \{(\lambda_1, \cdots, \lambda_\nu) \in \frak{t}^* ~|~ \lambda_1 \geq \cdots \geq \lambda_\nu \geq 0\},\,\, \mbox{where $(\lambda_1, \cdots, \lambda_\nu) := \lambda^{\vphantom{*}}_1h_1^*+ \cdots +  \lambda^{\vphantom{*}}_\nu h_\nu^*.$}
\end{equation}
Let $I^{(n)}_\lambda = I^{(2\nu + 1)}_\lambda$ be the skew-symmetric matrix in $\mcal{S}_n$ corresponding to $(\lambda_1, \cdots, \lambda_\nu)$, that is,
\begin{equation}\label{equ_Ilambdaodd}
I^{(n)}_\lambda :=
\begin{pmatrix}
		B({\lambda_1}) & \cdots & 0 & 0\\
		\vdots & \ddots & \vdots & \vdots  \\
		0  & \cdots & B({\lambda_\nu}) & 0 \\
		0         & \cdots & 0 & 0
	\end{pmatrix} \in \mcal{S}_{2\nu+1}
\end{equation}
where $\lambda = (\lambda_1, \cdots, \lambda_\nu) \in \frak{t}^*_+$. The following lemma is well-known.

\begin{lemma}\label{lemma_spectralodd}
For a real skew-symmetric matrix $A \in \mcal{S}_{n}$ where $n = 2\nu+1$,  there exists $Q \in \rmso(n)$ such that
\[
Q^T A Q = I^{(n)}_\lambda  
\]
for some $\lambda_1 \geq \lambda_2 \geq \cdots \geq \lambda_\nu \geq 0$ where $I^{(n)}_\lambda$ is in~\eqref{equ_Ilambdaodd}.
\end{lemma}

Lemma~\ref{lemma_spectralodd} tells us that the orbit of any skew-symmetric matrix under the conjugate $\mathrm{SO}(n)$-action intersects the positive Weyl chamber at 
exactly one point $(\lambda_1, \cdots, \lambda_\nu) \in \frak{t}^*_+$ in~\eqref{equ_positiveweylodd}. 
For any $\lambda = (\lambda_1, \cdots, \lambda_\nu) \in \frak{t}_+^*$, the \emph{co-adjoint orbit of}  $\mathrm{SO(n)}$ \emph{associated to} $\lambda$ is expressed as 
\begin{equation}\label{equ_coadjointorbit2nu+1}
\mcal{O}_\lambda^{(n)} = \left\{ Q^T \cdot I^{(n)}_\lambda \cdot Q \in \mcal{S}_{n} \mid Q \in \mathrm{SO}(n) \right\} = \left\{ A \in \mcal{S}_{n} \mid \textup{spec}(A) \equiv \langle \pm \lambda_1 i, \cdots, \pm \lambda_\nu i, 0 \rangle \right\}
\end{equation}
where $\textup{spec}(A) \equiv \langle \pm \lambda_1 i, \cdots, \pm \lambda_\nu i, 0 \rangle$ is the complete list of eigenvalues for $A$ including all repeated eigenvalues.  
Let us equip the co-adjoint orbit $\mcal{O}_\lambda^{(2\nu+1)}$ with the Kirillov--Kostant--Souriau form $\omega_\lambda^{(2\nu+1)}$. 

On the other hand, for the case where $G = \rmso(n)$ with $n := 2 \nu$, we consider the maximal torus 
\[
	T := \left\{ \begin{pmatrix}
		R(\theta_1) & \cdots & 0 \\
		\vdots & \ddots & \vdots \\
		0  & \cdots & R(\theta_\nu)
	\end{pmatrix} \in \rmso(2 \nu) ~:~  (\theta_1, \cdots, \theta_\nu) \in \R^\nu \right\}.
\]
The set of roots of $\frak{g}$ is given by $\{ \pm(h_i^* \pm h_j^*) \}_{i \neq j}$ and a base is taken by $\Delta = \{h_1^* - h_2^*, \cdots, h_{\nu-1}^* \pm h_\nu^* \}.$
The positive Weyl chamber with respect to $\Delta$ is then 
\begin{equation}\label{equ_coadjointorbit2nu}
	\frak{t}_+^* = \left\{(\lambda_1, \cdots, \lambda_\nu) \in \frak{t}^* ~|~ \lambda_1 \geq \cdots \geq \lambda_{\nu-1} \geq |\lambda_\nu| \right\}, \mbox{where $(\lambda_1, \cdots, \lambda_\nu) = \lambda_1h_1^* + \cdots + \lambda_\nu h_\nu^*$.} 
\end{equation}
Let $I^{(n)}_\lambda = I^{(2\nu)}_\lambda$ be the skew-symmetric matrix in $\mcal{S}_n$ corresponding to $(\lambda_1, \cdots, \lambda_\nu)$.
\begin{lemma}\label{lemma_spectraleven}
For a real skew-symmetric matrix $A \in \mcal{S}_{n}$ where $n = 2\nu$, there exists $Q \in \mathrm{SO}(n)$ such that
\[
Q^T A Q = I^{(n)}_\lambda
\]
for some $\lambda_1 \geq \cdots \geq \lambda_{\nu-1} \geq |\lambda_\nu|$.  
\end{lemma}

By Lemma~\ref{lemma_spectraleven}, the orbit of any skew-symmetric matrix under the conjugate $\mathrm{SO}(n)$-action intersects the positive Weyl chamber at 
exactly one point $(\lambda_1, \dots, \lambda_\nu) \in \frak{t}^*_+$ in~\eqref{equ_coadjointorbit2nu}. 
For any $\lambda = (\lambda_1, \dots, \lambda_\nu) \in \frak{t}_+^*$, the \emph{co-adjoint orbit of $\mathrm{SO}(n)$ associated to} $\lambda$ can be expressed as 
\begin{equation}\label{equ_coadjointorbit2nu}
\begin{split}
\mcal{O}_\lambda^{(n)} &= \left\{ Q^T \cdot I^{(n)}_\lambda \cdot Q \in \mcal{S}_{n} \mid Q \in \mathrm{SO}(n) \right\} \\
&= \left\{ A \in \mcal{S}_{n} \mid  \mathrm{pf}(A) =  \mathrm{pf}(I^{(n)}_\lambda) = \lambda_1 \dots \lambda_\nu, \, \textup{spec}(A) \equiv \langle \pm \lambda_1 \sqrt{-1}, \dots, \pm \lambda_\nu \sqrt{-1} \rangle \right\}
\end{split}
\end{equation}
where $\mathrm{pf}(A)$ is the Pfaffian of $A$ and $\textup{spec}(A) \equiv \langle \pm \lambda_1 i, \cdots, \pm \lambda_\nu i \rangle$ is the complete list of eigenvalues for $A$ including all repeated eigenvalues.
Let us equip $\mcal{O}_\lambda^{(2\nu)}$ with the Kirillov--Kostant--Souriau form $\omega^{(2\nu)}_\lambda$.

\subsection{Classical Gelfand--Cetlin systems on co-adjoint $\rmso(n)$-orbits}

Let $G = \rmso(n)$. For any $\lambda \in \frak{t}^*_+$, consider the co-adjoint orbit $(\mcal{O}^{(n)}_\lambda, \omega^{(n)}_\lambda )$ defined in~\eqref{equ_coadjointorbit2nu+1} and~\eqref{equ_coadjointorbit2nu}.
The inclusion $\mu \colon \mcal{O}^{(n)}_\lambda \to \frak{g}^*$ can be taken as a moment map of the co-adjoint $G$-action.
Thimm's trick from the sequence $G_\bullet$ of subgroups of $G$ in~\eqref{equation_nested_sequence} consisting of 
\[
G_{i-1} = \left\{
	\begin{pmatrix}
		B & 0 \\ 0 & I_{n-i} 
	\end{pmatrix} \mid B \in \rmso(i)
\right \}
\quad
\mbox{for $i=1,\cdots,n$}
\]
yields the set of collective Hamiltonian functions on (an open dense subset of) the Hamiltonian $G$-manifold $(\mcal{O}^{(n)}_\lambda, \omega_\lambda, \mu)$. 
We denote by $\Phi_{\lambda}^{(n)}$ the set of all non-constant such collective Hamiltonian functions of the form~\eqref{equ_compoGCsys} on $(\mcal{O}^{(n)}_\lambda, \omega^{(n)})$.

\begin{proposition}[cf. Section 4 in \cite{GS_Th}]\label{proposition_cisgc}
The system $\Phi_{\lambda}^{(n)}$ forms a completely integrable system (Definition~\ref{definition_cis}).
\end{proposition}

\begin{definition}
The system $\Phi_{\lambda}^{(n)}$ is called the \emph{(classical) Gelfand-Cetlin system} on $\mcal{O}^{(n)}_\lambda$. 
\end{definition}

The following proposition follows from the complete integrability of the (classical) Gelfand-Cetlin system. 

\begin{corollary}[cf. Section 4 in \cite{GS_Th}, Lemma 24 in \cite{Lane2}]\label{proposition_multiplicity_free}
	Let $G = \mathrm{SO}(n)$. Any co-adjoint $G$-orbit is a multiplicity free $\mathrm{SO}(n-1)$ space.
\end{corollary}

\begin{proof}
It follows from Theorem~\ref{theorem_multiplicityfree} and Proposition~\ref{proposition_cisgc}. 
\end{proof}

The remaining section is reserved for proving Proposition~\ref{proposition_cisgc}.
We shall describe the Gelfand--Cetlin system on $\mcal{O}^{(n)}_\lambda$ in terms of \emph{eigenvalues}. 
For any skew-symmetric matrix $A \in \mcal{S}_n$ in~\eqref{equ_coadjointorbit2nu+1} or~\eqref{equ_coadjointorbit2nu}, consider the $m$-th order leading principal submatrix $A^{(m)}$ of $A$ where $2 \leq m \leq n$.  
We then have the eigenvalues of $A^{(m)}$, which are all pure imaginary. Setting $\nu(m) :=  \lfloor m / 2 \rfloor$, we may arrange them in a \emph{descending} order as follow
\[
\begin{cases}
\left\langle \pm \lambda^{(m)}_1 \sqrt{-1},\,\, \pm \lambda^{(m)}_2 \sqrt{-1},\,\, \cdots,\,\, \pm \lambda^{(m)}_{\nu(m)} \sqrt{-1} \right\rangle, \quad &\mbox{if $m$ is even } \\
\left\langle \pm \lambda^{(m)}_1 \sqrt{-1},\,\, \pm \lambda^{(m)}_2 \sqrt{-1},\,\, \cdots,\,\, \pm \lambda^{(m)}_{\nu(m)} \sqrt{-1}, 0 \right\rangle,  \quad &\mbox{if $m$ is odd} 
\end{cases}
\]
where $\lambda^{(m)}_1 \geq \cdots \geq \lambda^{(m)}_{\nu(m) - 1} \geq \lambda^{(m)}_{\nu(m)} \geq 0$. 
For $1 \leq i \leq \nu(m)$, define 
\begin{equation}\label{equ_gccomponentmoded}
\Phi_\lambda^{i,m-i}\colon  \mcal{O}_\lambda^{(n)} \to \R 
\, \mbox{   by   } \,
A \mapsto \Phi_\lambda^{i,m-i} (A) := 
\begin{cases}
- \lambda_{\nu(m)}^{(m)} \quad &\mbox{if $i = \nu(m)$ and $\mathrm{pf}(A^{(m)}) < 0$} \\
+ \lambda_i^{(m)} \quad &\mbox{otherwise.}
\end{cases}
\end{equation}
By Lemma~\ref{lemma_spectralodd} and~\ref{lemma_spectraleven},
the collection of non-constant functions in $\left\{ \Phi_\lambda^{i, m-i} \mid 2 \leq m < n, 1 \leq i \leq \nu(m) \right\}$ is exactly the GC system $\Phi_{\lambda}^{(n)}$ on $(\mcal{O}^{(n)}_\lambda, \omega_\lambda, \mu)$  under the identification~\eqref{equ_dualformidentification}. 
This collection will be denoted by
\begin{equation}\label{equ_GCsystemsolambdabd}
\Phi_\lambda^{(n)} \colon \mcal{O}_\lambda^{(n)} \to \R^{\dim_\R \mcal{O}_\lambda^{(n)}/2}.
\end{equation}

In the case where $G =\rmso(n)$, the image of a Gelfand--Cetlin system $\Phi_\lambda^{(n)}$ in~\eqref{equ_GCsystemsolambdabd} is a convex polytope, called a {\em Gelfand--Cetlin polytope} or simply a {\em GC polytope} which we denote by $\Delta_\lambda^{(n)}$. 
The polytope is indeed defined by intersecting the inequalities from the min-max principle of skew-symmetric matrices. Let $(u_{i,j})$ be the coordinates of $\R^{\dim_\R \mcal{O}_\lambda^{(n)}}$ recording the output of $\Phi_\lambda^{i,j}$ in~\eqref{equ_gccomponentmoded}. In terms of the coordinates $(u_{i,j})$, the polytope $\Delta_\lambda^{(n)}$ is described as follows.

\begin{theorem}[\cite{Lit, Pab}]\label{theorem_GCpolytopeineq}
	For $G = \rmso(n)$ and $\lambda = (\lambda_1, \cdots, \lambda_\nu) \in \frak{t}_+^*$, setting $u_{j,n-j} := \lambda_j$ for $j = 1, \cdots, \nu =  \lfloor n / 2 \rfloor$,
	the Gelfand--Cetlin polytope $\Delta_\lambda^{(n)}$ in $\R^{\dim_\R \mcal{O}_\lambda^{(n)}/2}$
	consists of points $(u_{i,j})$ such that 
	\begin{equation}\label{equ_minmax_son}
		\begin{cases}
			u_{i, j+1} \geq u_{i,j}  \quad \text{if} \quad i < j, &\quad u_{i, j+1} \geq |u_{i,j}| \quad \text{if} \quad i = j \\
			u_{i, j} \geq u_{i+1,j}  \quad \text{if} \quad i+1 < j, &\quad u_{i, j} \geq |u_{i+1,j}| \quad \text{if} \quad i+1 = j 
		\end{cases}
	\end{equation}
	for $(i,j) \in \mathbb{N} \times \mathbb{N}$ with $2 \leq i+j \leq n$ and $i \leq j$.
\end{theorem}

\begin{example}
When $n = 4$ and $5$, 
\begin{itemize}
\item $\Delta_\lambda^{(4)}$ is determined by $\lambda_1 = u_{1,3} \geq u_{1,2} \geq | u_{1,1} |, \, u_{1,2} \geq |u_{2,2}| = \lambda_2$, 
\item $\Delta_\lambda^{(5)}$ is determined by $\lambda_1 = u_{1,4} \geq u_{1,3} \geq u_{1,2} \geq | u_{1,1} |, \, u_{1,3} \geq \lambda_2 = u_{2,3} \geq | u_{2,2} |,\, u_{1,2} \geq | u_{2,2} |$.
\end{itemize}
\end{example}

\begin{remark}
The min-max principle yields that the image of $\Phi_\lambda^{(n)}$ in~\eqref{equ_GCsystemsolambdabd} is contained in $\Delta_\lambda^{(n)}$. 
For the reverse inclusion, the reader is referred to Lemma 4.1 and 4.2 in \cite{Pab2}.
\end{remark}

\begin{proof}[Proof of Proposition~\ref{proposition_cisgc}]
As in Section~\ref{ssecCollectivePeriodicHamiltonians} or \cite[Proposition 3.1]{GS_GC}, any pair of collective Hamiltonian functions on $\mcal{O}^{(n)}_\lambda$ arising from  $G$-invariants functions on $\frak{g}^*$ is Poisson commutative with respect to $\{\cdot, \cdot \}_M$. The functional independence of $\Phi_\lambda^{(n)}$ follows from $\dim \Delta_\lambda^{(n)} = \dim_\R \mcal{O}_\lambda^{(n)}/2$.
\end{proof}

\vspace{0.1cm}
\section{The face structure of Gelfand--Cetlin polytopes}
\label{secTheFaceStructureOfGelfandCetlinPolytopes}

The goal of this section is to describe the face structure of the GC polytope $\Delta_\lambda^{(n)}$ in terms of a certain partially ordered set of pairs of subgraphs of the associated ladder diagram. 
 
For a positive integer $n \in \Z_{>0}$, set $\nu := \lfloor n / 2 \rfloor$. Let $\lambda = \{ \lambda_1, \cdots, \lambda_\nu \}$ be a $\nu$-tuple of real numbers such that 
\begin{equation}\label{equ_sequoflambbda}
\lambda_1 = \cdots = \lambda_{n_1} > \lambda_{n_1+1} = \cdots = \lambda_{n_2} > \cdots > \lambda_{n_r + 1} = \cdots = \lambda_{n_{r+1}} (=\lambda_\nu) \geq 0
\end{equation}
for some positive integers $n_1 < \cdots < n_r < n_{r+1} = \nu$. 
For a pair $(n, \lambda)$, 
we define the function $\iota_{n,\lambda}$ by
\[
\iota_{n,\lambda} \colon \left\{1, \cdots, \nu \right\} \to \left\{ 1, \cdots, r+1 \right\}, \quad
\iota_{n, \lambda} (j) := i
\] 
where $i$ is the index determined by $\lambda_j = \lambda_{n_i}$.
For simplicity, the sub-index $(n,\lambda)$ for the function $\iota_{n,\lambda}$ will be suppressed, i.e., 
$\iota(j) := \iota_{(n,\lambda)}(j)$ if there is no danger of confusion.
Let $n(\lambda)$ be the number of \emph{nonzero} elements of $\lambda$, that is 
\[
n(\lambda) :=
\begin{cases}
n_r \quad &\mbox{if $\lambda_\nu = 0$}\\
n_{r + 1} (= \nu) \quad &\mbox{if $\lambda_\nu > 0$}.
\end{cases}
\] 

We define the ladder diagram associated to the pair $(n, \lambda)$ as follows. 

\begin{definition}\label{def_ladder}
Let $\Gamma_{\Z^2}$ be the square grid graph, $\Gamma_{\Z^2} := (\Z \times \R) \cup (\R \times \Z).$
The \emph{ladder diagram $\Gamma^{(n)}_\lambda$ associated to $(n, \lambda)$} is defined by the induced subgraph of $\Gamma_{\Z^2}$ formed from the following set of vertices
\[
	V_{\Gamma^{(n)}_\lambda} := \bigcup_{j=0}^{n(\lambda) - 1}  \left\{ (a, b) \in \Z^2 \,~\Big{|}~\, j \leq a \leq j+1, j \leq b \leq n - 1 - n_{\iota(j+1)} \right\}.
\]
\end{definition}

The vertex located at $(0,0)$ is called the \emph{origin}. A vertex which is farthest from the origin with respect to the taxicap metric on $\Gamma_{\Z^2}$ is called a \emph{terminal vertex} of $\Gamma^\bullet_\lambda$. A \emph{lowest terminal vertex} of $\Gamma^\bullet_\lambda$ is the (unique) terminal vertex $(a_0,b_0)$ such that $b_0 \leq b$ for any terminal vertex $(a, b)$ of $\Gamma^\bullet_\lambda$. We denote by $\square^{(i,j)}$ the unit square whose top-right vertex is $(i,j)$ in $\Gamma_{\Z^2}$.

\begin{example}
Figure~\ref{Fig_Ladder_dia} illustrates the ladder diagrams $\Gamma^{(7)}_{3,2,1}, \Gamma^{(7)}_{2,2,0}, \Gamma^{(8)}_{5,3,3,1}$, and $\Gamma^{(8)}_{4,4,2,0}$. The purple points are the origins and the blue points are the lowest terminal vertices. The green points are the other terminal vertices. 
\begin{figure}[h]
	\scalebox{0.85}{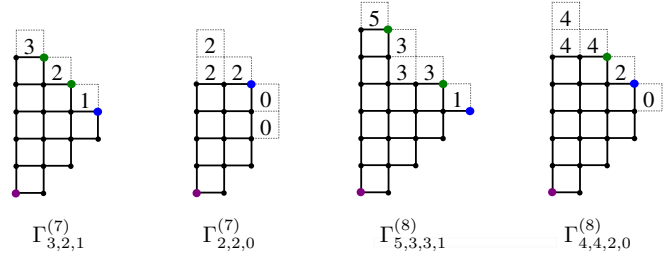}
	\caption{\label{Fig_Ladder_dia} Examples of Ladder diagrams.}	
\end{figure}
\end{example}

To describe the face structure of a GC polytope,  we need the following notions.

\begin{definition}\label{def_combinatorialobjects}
Let $\Gamma^{(n)}_\lambda$ be a ladder diagram.
\begin{itemize}
\item A \emph{positive path} on $\Gamma^{(n)}_\lambda$ is a shortest path from the origin to a terminal vertex in $\Gamma^{(n)}_\lambda$. 
\item An \emph{isogram} $\gamma$ is the subgraph of $\Gamma^{(n)}_\lambda$ obeying the following.
\begin{enumerate}
\item The graph $\gamma$ can be presented as a union of positive paths containing all terminal vertices of $\Gamma^{(n)}_\lambda$. 
\item 
Once $\gamma$ contains \emph{both} the top edge and the left edge of a  box $\square^{(j,j)}$ on the diagonal, 
 its bottom edge and right edge must be also contained in $\gamma$. 
Although the pattern in Figure~\ref{Fig_Cond_Isogram} can be a part of a union of positive paths at $\square^{(j,j)}$, it is \emph{not} allowed for the subgraph being an isogram.

\begin{figure}[h]
	\scalebox{1.2}{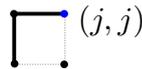}
		\vspace{-0.2cm}
	\caption{\label{Fig_Cond_Isogram} Forbidden pattern for isograms at the diagonal box $\square^{(j,j)}$.}	
\end{figure}
\end{enumerate}
\item For an isogram $\gamma$ and a non-negative integer $j$ satisfying $0 \leq j < n(\lambda)$, consider the horizontal segments of $\gamma$  over the line segment $\{ a \in \R ~|~ j \leq a \leq j+1\}$. 
Let $\beta_j(\gamma)$ be the \emph{lowest} horizontal line segment of $\gamma$ among the horizontal segments.
Then $\beta_j(\gamma)$ can be expressed as 
\begin{equation}\label{equ_beta2}
\beta_j(\gamma) = \{ (i, y(\beta_{j}(\gamma)) ) \mid j \leq i \leq j + 1 \}
\end{equation}
where $y(\beta_{j}(\gamma)) \in \mathbb{N} \cup \{ 0\}$ is the second component of $\beta_j(\gamma)$.
The \emph{base} $\beta(\gamma)$ of an isogram $\gamma$ is defined to be a (unique) positive path containing all $\beta_j(\gamma)$'s.
Thus, the base $\beta_j(\gamma)$ is the {\em lowest} positive path contained in $\gamma$.

\item For each non-negative integer $j$ with $0 \leq j < n(\lambda)$, we choose a line segment or a union of line segments (say $\delta_j(\gamma)$) obeying the following rule. 
\begin{enumerate}
\item 
For each $j$ with $y(\beta_{j}(\gamma)) > j+1$, we take 
\[
	\delta_j(\gamma) := \beta_j(\gamma).
\]
\item If $\gamma$ contains all four edges of $\square^{(j+1,j+1)}$,
we take two horizontal edges of $\square^{(j+1,j+1)}$ as $\delta_j(\gamma)$, i.e., 
\[
\delta_j(\gamma) := \beta_j(\gamma) \cup ( \beta_j (\gamma) + (0,1)).
\]
\item If we are not in the above two cases{\footnote{In other words, the case (3) happens when the second component $y(\beta_{j}(\gamma))$ is $\leq j +1$ and $\gamma$ does \emph{not} contain all four edges of $\square^{(j+1,j+1)}$.}},
we have two options to choose
\[
\delta_j(\gamma) := 
\begin{cases} 
\,\, \left\{ (a,b) \in \R^2 ~{|}~ j \leq a \leq j+1,  b = j \right\} \textup{ or } \\
\,\, \left\{ (a,b) \in \R^2 ~{|}~ j \leq a \leq j+1,  b = j+1 \right\}.
\end{cases}
\]
\end{enumerate}
A \emph{coastline} $\delta$ of $\gamma$ is the graph which can be expressed as the union of all positive paths whose horizontal edges are contained in the set of chosen horizontal segments $\delta_j (\gamma)$'s.
\end{itemize}
\end{definition}

\begin{remark}
Even though two points satisfy the same inequalities in~\eqref{equ_minmax_son}, the diffeomorphism types of the fibers over the two points might be different. The base lines play  important roles to distinguish their diffeomorphism types later on. 
Also, coastlines are necessary to obtain the one-to-one correspondence. There exist multiple faces corresponding to the same isogram and coastlines are used to distinguish the faces having the same isogram.
\end{remark}

\begin{definition}
A \emph{face} of $\Gamma^{(n)}_\lambda$ is a pair consisting of an isogram $\gamma$ and one single choice of its coastline $\delta(\gamma)$. The {\em dimension} of $\gamma$ is defined to be the number of closed regions bounded by $\gamma$, the number of minimal cycles in $\gamma$.
\end{definition}

\begin{example}\label{exa_isograms}
Let us reconsider the ladder diagram $\Gamma^{(7)}_{3,2,1}$ in Figure~\ref{Fig_Ladder_dia} as an example. The first two graphs in Figure~\ref{Fig_Isograms} are examples of isograms. The other graphs are \emph{not} isograms. 
For instance, the last graph has the forbidden pattern at $\square^{(1,1)}$ (see Figure~\ref{Fig_Cond_Isogram}) so that it is \emph{not} an isogram. 
The first isogram in Figure~\ref{Fig_Isograms} comes with eight possible choices of coastlines. The second isogram in Figure~\ref{Fig_Isograms} admits two possible coastlines, which are depicted in Figure~\ref{Fig_Faces}.
\end{example}

\begin{figure}[h]
		\vspace{-0.5cm}
	\scalebox{0.9}{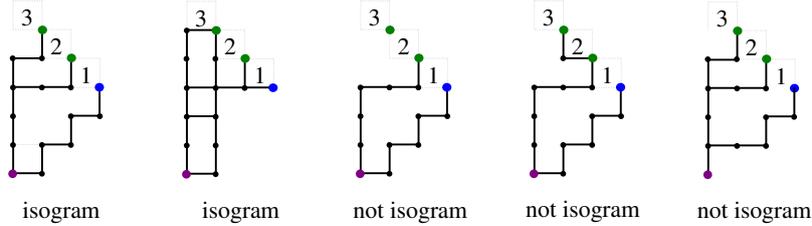}
		\vspace{-0.2cm}
	\caption{\label{Fig_Isograms} Examples of isograms.}	
\end{figure}

\begin{figure}[h]
		\vspace{-0.3cm}
	\scalebox{0.9}{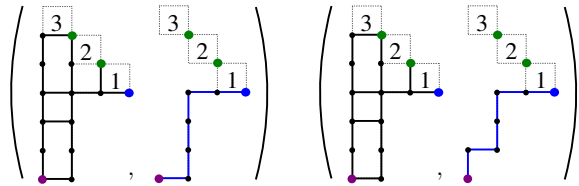}
			\vspace{-0.3cm}
	\caption{\label{Fig_Faces} Examples of faces.}	
\end{figure}

Define a partial ordering $\leq$ on the set of faces of $\Gamma^{(n)}_\lambda$ as follows. $(\gamma^\prime, \delta(\gamma^\prime)) \leq (\gamma, \delta(\gamma))$ if the following two conditions hold.
\begin{enumerate}
\item (Isogram) $\gamma^\prime \subseteq \gamma$
\item (Coastline) For each $j$ with $y(\beta_{j}(\gamma))  \leq j + 1$ in~\eqref{equ_beta2}, $\delta_j(\gamma^\prime) \subseteq \delta_j(\gamma).$
\end{enumerate}
Here, $\subseteq$ is meant to be a subset. 
Also, consider the partial ordering $\leq$ on the set of faces of $\Delta_\lambda^{(n)}$ defined by the set-theoretic inclusion. That is, for two faces $f_1$ and $f_2$ of $\Delta_\lambda^{(n)}$, $f_1 \leq f_2$ if and only if $f_1 \subseteq f_2$ as a set. 

Those faces of $\Gamma^{(n)}_\lambda$ can be used to understand the face structure of the GC polytope $\Delta_\lambda^{(n)}$.

\begin{theorem}\label{thm_diacorrespo}
Let $\Delta_\lambda^{(n)}$ be the Gelfand--Cetlin polytope and $\Gamma_\lambda^{(n)}$ the ladder diagram associated to $n$ and $\lambda$.
Then there exists a bijective map
	\[
			\Psi \colon \{~\text{faces of}~\Gamma_\lambda^{(n)} \} \longrightarrow \{~\text{faces of} ~\Delta_\lambda^{(n)} \}
	\]
such that for faces $(\gamma, \delta(\gamma))$ and $(\gamma', \delta(\gamma'))$ of $\Gamma_\lambda^{(n)}$,
\begin{itemize}
	\item (Order-preserving) $(\gamma, \delta(\gamma)) \leq (\gamma', \delta(\gamma'))$ if and only if $\Psi((\gamma, \delta(\gamma))) \leq \Psi((\gamma', \delta(\gamma')))$,
	\item (Dimension) $\dim \Psi((\gamma), \delta(\gamma)) = \dim \gamma$.
\end{itemize}
\end{theorem}

\begin{example}\label{exa_mcalobface} 
Let $n = 5$ and $\lambda = (3, 0)$. Consider the co-adjoint orbit $\mcal{O}^{(5)}_\lambda$, the orthogonal Grassmannian $\mathrm{OG}(1, \C^5)$. Then, the GC polytope $\Delta_\lambda^{(5)}$ is defined by the intersection of the four half-planes in $\R^4$ $\colon$ 
$$
3 - u_{1,3} \geq 0, \, u_{1,3} - u_{1,2} \geq 0, \, u_{1,2} - u_{1,1} \geq 0, \, u_{1,2} + u_{1,1} \geq 0, 
$$
which is a simplex as in Figure~\ref{Fig_OG15}. 

\begin{figure}[h]
	\scalebox{0.87}{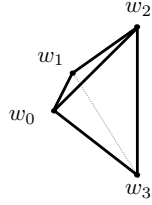}
	\caption{\label{Fig_OG15} The GC polytope $\Delta^{(5)}_{(3,0)}$}	
\end{figure}

It has four vertices $w_0 = (u_{1,3} = 0, u_{1,2} = 0, u_{1,1} = 0), w_1 = (3,0,0), w_2 = (3,3,3), w_3 = (3,3, -3)$ corresponding to the following pairs of an isogram and a coastline as in Figure~\ref{Fig_0OG15}.

\begin{figure}[h]
				\vspace{-0.1cm}
	\scalebox{0.6}{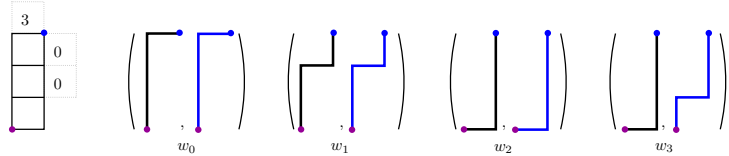}
				\vspace{-0.2cm}
	\caption{\label{Fig_0OG15} The zero dimensional faces of $\Delta^{(5)}_{(3,0)}$}	
				\vspace{-0.2cm}
\end{figure}

The polytope has six edges, four two dimensional faces, and one three dimensional face, whose corresponding faces in the diagram are in order. 

\begin{figure}[h]
	\scalebox{0.6}{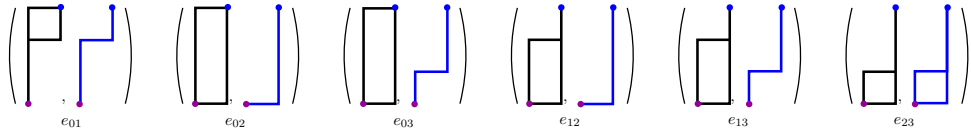}
					\vspace{-0.2cm}
	\caption{\label{Fig_1OG15} The one dimensional faces of $\Delta^{(5)}_{(3,0)}$}	
					\vspace{-0.2cm}
\end{figure}
\begin{figure}[h]
	\scalebox{0.6}{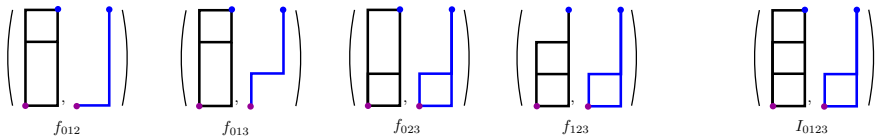}
					\vspace{-0.2cm}
	\caption{\label{Fig_23OG15} The two and three dimensional faces of $\Delta^{(5)}_{(3,0)}$}	
					\vspace{-0.2cm}
\end{figure}
\end{example}

Theorem \ref{thm_diacorrespo} is an $\rmso(n)$-type analog of \cite[Theorem 1.11]{ACK} and the idea of the proof is essentially the same as the one of
\cite[Theorem 1.11]{ACK}. Thus we leave the proof to the reader with some remarks as below. 

\begin{itemize}
	\item As mentioned in Theorem \ref{theorem_GCpolytopeineq}, the GC polytope $\Delta_\lambda^{(n)}$ can be described in terms of the coordinate system $\{u_{i,j} \}$ 
	         satisfying \eqref{equ_minmax_son} where each $u_{i,j}$ corresponds to the unit box $\square^{(i,j)}$ in $\Gamma_\lambda^{(n)}$. 
	\item We can describe each point in $\Delta_\lambda^{(n)}$ as the filling in the boxes of $\Gamma_\lambda^{(n)}$ with the components $v_{i,j}$ where
	\[
		v_{i,j} = 
			\begin{cases}
				u_{i,j} \quad &\mbox{if $i \neq j$} \\
				|u_{i,j}| \quad &\mbox{if $i = j$}
			\end{cases}
	\] 
	For any given face $\Gamma = (\gamma, \delta)$ in $\Gamma^{(n)}_\lambda$, the corresponding face is supported by the intersection of the following$\colon$ 
	\[
		\begin{cases}
			v_{i,j} = v_{i, j+1} \quad &\mbox{if the isogram $\gamma$ does not contain the edge connecting $(i-1,j)$ and $(i, j)$. }\\
			v_{i,j} = v_{i+1, j} \quad &\mbox{if the isogram $\gamma$ does not contain the edge connecting $(i,j)$ and $(i, j-1)$.} \\
			u_{i,j} \geq 0 \quad &\mbox{if $\square^{(i,j)}$ is above the highest positive path contained in the coastline $\delta$.} \\
			u_{i,j} \leq 0 \quad &\mbox{if $\square^{(i,j)}$ is below the lowest positive path contained in the coastline $\delta$.} \\
		\end{cases}
	\]
\end{itemize}

\vspace{0.1cm}
\section{The topology of Gelfand--Cetlin fibers}
\label{secTheTopologyOfGelfandCetlinFibers}

This section is devoted to describing a Gelfand--Cetlin fiber of a co-adjoint $\mathrm{SO}(n)$-orbit in terms of the total space of an iterated bundle described in Theorem \ref{thmx_main}.
For a point $\bf{u}$ in $\Delta_\lambda^{(n)}$, the isogram $\gamma$ corresponding to the face containing $\bf{u}$ in its relative interior is obtained by the correspondence in 
Theorem \ref{thm_diacorrespo}.
We will devise a constructive combinatorial method to decode the topology of fibers at stages of the iterated bundle from $\gamma$.

The main theorem is stated as follows.

\begin{theorem}\label{thm_fiber}
	Let $\Phi_\lambda^{(n)}$ be the Gelfand--Cetlin system on
	the co-adjoint $\mathrm{SO}(n)$-orbit $\mathcal{O}^{(n)}_\lambda$ for $\lambda \in \mathfrak{t}^*_+$ and $\Delta^{(n)}_\lambda$ the corresponding Gelfand--Cetlin
	polytope. For any point ${\bf{u}} \in \Delta_\lambda^{(n)}$, the fiber 
	$\left(\Phi_\lambda^{(n)} \right)^{-1}({\bf{{u}}})$ is an isotropic submanifold of $\left( \mcal{O}_\lambda^{(n)}, \omega_\lambda \right)$ and is the
	total space of an iterated bundle
	\begin{equation}\label{equ_iteratededed}
		\left(\Phi_\lambda^{(n)} \right)^{-1}({\bf{{u}}}) = E_{n} \stackrel{q_{n}} \longrightarrow E_{n-1} \stackrel{q_{n-1}} \longrightarrow \cdots \stackrel{q_4} \longrightarrow E_3
		\stackrel{q_3} \longrightarrow E_2= \{\mathrm{point}\}
	\end{equation}
	such that the fiber at each stage is either a point or a product of spheres. 
	Moreover, any two fibers $\left(\Phi_\lambda^{(n)} \right)^{-1}({\bf{{u}}}_1)$ and $\left(\Phi_\lambda^{(n)} \right)^{-1}({\bf{{u}}}_2)$ are diffeomorphic if $\bf{{u}}_1$ and $\bf{{u}}_2$ are contained in the relative interior of the same face of $\Delta_\lambda^{(n)}$.
\end{theorem}

 The proof will be given throughout Section~\ref{secIteratedBundleStructureOfGelfandCetlinFibers} and~\ref{sectionproofoflemma}.

\begin{remark}
The authors with Oh in \cite[Theorem A]{CKO} proved that every GC fiber of type $A$ has an iterated bundle structure and is an isotropic submanifold. Lane \cite{Lane2} described GC fibers of Hamiltonian $\mathrm{U}(n)$-manifolds (Hamiltonian $\mathrm{SO}(n)$-manifolds as well) as the total space of an iterated bundle.
We remark that his description is slightly different from one in Theorem~\ref{thm_fiber}$\colon$ the iterated bundle in \cite{Lane2} has a torus base and does not have circle factors in any of the fibers at stages.
Also, Bouloc--Miranda--Zung in \cite{BMZ} proved that any GC fiber of type $A$ is an isotropic submanifold and expressed the fibers as a 2-stage quotient of some compact Lie group (which 
they call a {\em coarse symmetry group}). 
\end{remark}

\begin{remark}\label{rmk_fibertypeseven}
Note that the fiber in a stage of~\eqref{equ_iteratededed} for the $\mathrm{SO}(n)$-type can have an \emph{even}-dimensional spherical factor, while either a point or a product of odd-dimensional spheres can only occur in the stages of the iterated bundles for the GC fibers in the co-adjoint $\mathrm{U}(n)$-orbits (cf. \cite{CKO})
\end{remark}

We begin by introducing blocks which will be used for our combinatorial process.
For each lattice point $(a, b) \in \Z^2$, let $\blacksquare^{(a,b)}$ be the closed region bounded by $\square^{(a,b)}$, that is,
\[
\blacksquare^{(a,b)} := \{ (x,y) \in \R^2 ~|~ a-1 \leq x \leq a, ~b-1 \leq y \leq b \}.
\] 

\begin{definition}
For each $k \in \N_{\geq 3}$, the {\em $W_k$-block} is the closed region defined by 
\begin{equation}\label{eq_wkunion}
W_k =  \bigcup_{(a,b)} {\blacksquare}^{(a,b)}
\end{equation}
where the union is taken over all $(a,b) \in \mathbb{N}^2$ satisfying
$k - 1 \leq a+ b \leq k$ and $a \leq b$.

A lattice point closest from the origin in the block $W_k$ is called a \emph{bottom point}. 
\end{definition}

$W$-blocks are depicted in Figure~\ref{Fig_Wblocks} where the red points are the origin, the purple points are bottom points, and the blue points indicate the vertices over which the union is taken in~\eqref{eq_wkunion}. 

\begin{figure}[h]
	\scalebox{0.7}{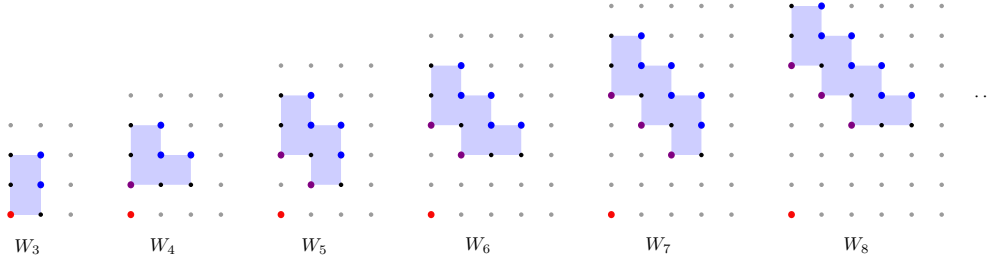}
	\vspace{-0.3cm}
	\caption{\label{Fig_Wblocks} $W$-blocks.}	
\end{figure}

\begin{definition} Let $k \in \N$ be a given positive integer.
\begin{enumerate}
\item (Definition 5.7. in \cite{CKO}) For each $(a_0, b_0) \in \Z^2$, the \emph{$M_k$-block} at $(a_0, b_0)$
is the closed region defined by
\begin{equation*}\label{eq_mkunion}
M_k (a_0, b_0) := \bigcup_{(a_0, b_0) + (a,b)} {\blacksquare}^{(a,b)}
\end{equation*}
for some $(a_0, b_0) \in \Z^2$ where the union is taken over all $(a,b) \in \mathbb{N}^2$ satisfying
$$k + 1 \leq a+ b \leq k + 2, \, (a,b) \neq (k+1, 1), \mbox{ and } (a,b) \neq (1, k+1).$$

\item An \emph{$N_k$-block} is the closed region defined by
\begin{equation*}\label{eq_mkunion}
N_k (a_0, b_0) :=  \bigcup_{(a_0, b_0) + (a,b)} {\blacksquare}^{(a,b)}
\end{equation*}
for some $(a_0, b_0) \in \Z^2$ where the union is taken over all $(a,b) \in \mathbb{N}^2$ satisfying
$$k + 1 \leq a+ b \leq k + 2 \mbox{ and } (a,b) \neq (1, k+1).$$
\end{enumerate}
\end{definition}

\begin{figure}[h]
\vspace{0.1cm}
	\scalebox{0.82}{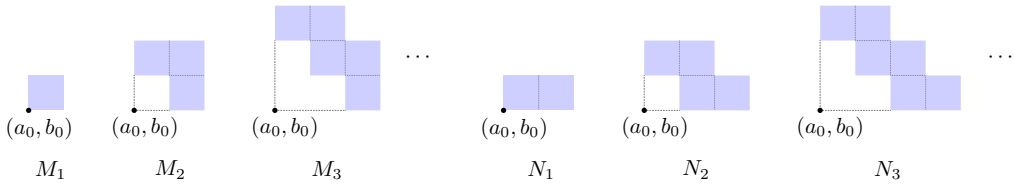}
	\vspace{-0.2cm}
	\caption{\label{Fig_MNblocks} $M$-blocks and $N$-blocks.}	
\end{figure}

For a given point ${\bf{u}} \in \Delta_\lambda^{(n)}$, let $f$ be the face of $\Delta_\lambda^{(n)}$ containing $\bf{u}$ in its relative interior and 
$(\gamma_f, \delta_f)$ the corresponding face of $\Gamma^{(n)}_\lambda$.  
Now we describe fibers at stages for an iterated bundle structure of $\left(\Phi_\lambda^{(n)}\right)^{-1}({\bf u})$ in the following steps. 
(Compare with the $\rmu(n)$-cases illustrated in \cite[Section 5]{CKO}.)
\begin{itemize}
\item \textbf{Step 1.} Overlapping the $W$-block $W_k$ and the isogram $\gamma_f$, cut $W_k$ along $\gamma_f$ into several closed regions.
\item \textbf{Step 2.} For each cut region $\mcal{D}$, assign a topological space $S_k(\mcal{D})$ as follows$\colon$
\begin{equation}\label{equ_fibersfromblocks}
S_k(\mcal{D}) := 
\begin{cases}
\mathbb{S}^{2\ell - 1} \quad &\mbox{ if $\mcal{D}$ is an $M_\ell$-block and $\mcal{D}$ contains a bottom point of $W_k$} \\
\mathbb{S}^{2\ell } \quad &\mbox{ if $\mcal{D}$ is an $N_\ell$-block and $\mcal{D}$ is below the base $\beta(\gamma_f)$.} \\
\textup{point} \quad &\mbox{otherwise.}
\end{cases}
\end{equation}
\item \textbf{Step 3.} Put
\begin{equation}\label{equ_fibersatkthstage}
S_k(f) := \prod_\mcal{D} S_k (\mcal{D})
\end{equation}
where the product is over all cut closed regions of $W_k$ along $\gamma_f$.
\end{itemize}
Following the three steps above, the total space of the bundle, the GC fiber over $\bf{{u}}$, can be described as follows.

\begin{proposition}\label{proposition_mainiterated}
Let $\bf{{u}}$ be a point contained in the relative interior of a face $f$ of $\Delta_\lambda^{(n)}$. Let $(\gamma_f, \delta_f)$ be the corresponding face to $f$ in $\Gamma_\lambda^{(n)}$. 
Then the fiber $\left(\Phi_\lambda^{(n)}\right)^{-1}(\bf{{u}})$ is diffeomorphic to the total space of an iterated bundle
\begin{equation}\label{equ_iteratedbundlestr}
\left(\Phi_\lambda^{(n)}\right)^{-1}({\bf{u}} ) = \overline{S_{n}}(f) \xrightarrow{q_{n}} \overline{S_{n-1}}(f) \xrightarrow{} \cdots \xrightarrow{q_{4}} \overline{S_{3}}(f) \xrightarrow{q_{3}} \overline{S_{2}}(f) := \{\textup{point}\}
\end{equation}
where $q_k \colon \overline{S_{k}}(f) \xrightarrow{} \overline{S_{k-1}}(f)$ is an $S_k(f)$-bundle over $\overline{S_{k-1}}(f)$. 
In particular, the dimension of the fiber $\left(\Phi_\lambda^{(n)} \right)^{-1}(\bf{{u}})$ can be computed as 
\begin{equation}\label{eq_dimformulaprop}
\dim \left(\Phi_\lambda^{(n)}\right)^{-1}({\bf{u}}) = \sum_{k=3}^{n} \dim S_k(f).
\end{equation}
\end{proposition}

In particular, Proposition~\ref{proposition_mainiterated} describes the iterated bundle structure of a given GC fiber.
Proposition~\ref{proposition_mainiterated} will be proven in Section~\ref{subsection_fibersatstages}.
We close this section by giving an example demonstrating the above combinatorial process.

\begin{example}
	Let us revisit Example~\ref{exa_mcalobface}. Take the vertex $v = (0,0,0)$ for example. 
	Then the corresponding isogram $\gamma_v$ is the first graph in Figure~\ref{Fig_Fiberreading}. 
	We now follow the three steps to compute the fiber of $\Phi_\lambda^{(5)}$ at $v$.
	In the first and second steps, $W_3$ and $W_4$-blocks are \emph{not} cut into smaller pieces so that $S_3(f) \simeq S_4(f) \simeq \{ \mbox{point} \}$. 
	At the final step, $W_5$ is cut into two blocks, one is an $M_1$-block and the other is an $M_2$-block. 
	But, the $M_1$-block does not contain any bottom point of $W_5$ so that it does \emph{not} contribute to $S_5(f)$. We then have $S_5(f) \simeq \mathbb{S}^3$ 
	by \eqref{equ_fibersatkthstage}
	and
	\[
		\left( \Phi_\lambda^{(5)} \right)^{-1}(v) \cong \mathbb{S}^3.
	\] Moreover, the fiber is Lagrangian. 
	
\begin{figure}[h]
	\scalebox{0.7}{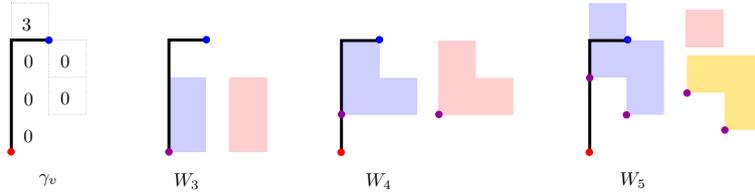}
	\vspace{-0.2cm}
	\caption{\label{Fig_Fiberreading} The GC fiber over $v = (0,0,0)$.}	
\end{figure}

Now, consider the edge determined by $u_{1,2} = u_{1,1} = 0$ of $\Delta_\lambda$ where the corresponding isogram is depicted in Figure \ref{Fig_Fiberreading2}.
Choose any point $w$ in the relative interior of the edge. 
The coordinate of $w$ is of the form $(u_{1,3}, u_{1,2}, u_{1,1}) = (a, 0, 0)$ for $a$ with $0 < a < 3$. 
Notice that the corresponding isogram $\gamma_w$ is given in Figure~\ref{Fig_Fiberreading2}.
Following the three steps above, we similarly obtain that
$S_3(f) \simeq  \{ \mbox{point} \}, \, S_4(f) \simeq \mathbb{S}^2,$ and  $\, S_5(f) \simeq \mathbb{S}^1.$
Thus, the GC fiber over $w$ is an $\mathbb{S}^1$-bundle over $\mathbb{S}^2$, which is also Lagrangian. 

\begin{figure}[h]
	\scalebox{0.7}{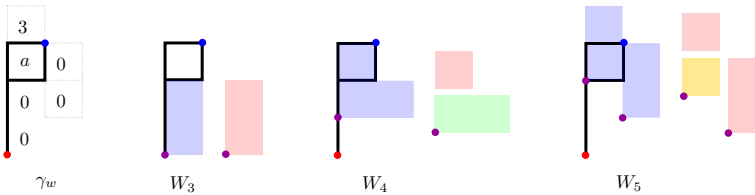}
	\vspace{-0.2cm}
	\caption{\label{Fig_Fiberreading2} The GC fiber over $w = (a,0,0)$ for some $a$ with $0 < a < 3$ .}	
\end{figure}
\end{example}

\section{Iterated bundle structure of Gelfand--Cetlin fibers}
\label{secIteratedBundleStructureOfGelfandCetlinFibers}

In this section, we will verify Proposition~\ref{proposition_mainiterated} and show that every GC fiber is isotropic.
Throughout this section, we denote by
\begin{itemize}
	\item $T(m)$ the maximal torus of $\rmso(m)$ defined in Section \ref{ssecRmsoNCases},
	\item $\frak{t}^*(m)$ the dual of the Lie algebra of $T(m)$,
	\item $\frak{t}^*_+(m)$ the positive Weyl chamber.
\end{itemize}	 

\subsection{Bundle structures at stages}\label{subsection_bundlestructureatstages}

Under the following identifications as in Section \ref{secGelfandCetlinSystems}
\begin{equation}\label{equ_t+*mt*m}
\begin{cases}
\frak{t}_{\vphantom{+}}^*(m) \cong \R^\ell, \,\, \frak{t}_+^*(m) \cong \left\{(\lambda_1, \cdots, \lambda_\ell) \in \R^\ell \mid \lambda_1 \geq \cdots \geq \lambda_\ell \geq 0 \right\} \quad &\mbox{if $m = 2\ell + 1$,} \\
\frak{t}_{\vphantom{+}}^*(m) \cong \R^{\ell+1}, \,\, \frak{t}_+^*(m) \cong \left\{(\lambda_1, \cdots, \lambda_{\ell+1}) \in \R^{\ell+1} \mid \lambda_1 \geq \cdots \geq \lambda_{\ell} \geq | \lambda_{\ell+1}| \right\} \quad &\mbox{if $m = 2\ell + 2$,} \\
\end{cases}
\end{equation}
let us choose two elements $\frak{a}$ and $\frak{b}$:
\begin{itemize}
\item \textbf{(Case 1)} when $m = 2 \ell + 1$, let
\[
	\begin{cases}
		\frak{a} := (a_1, \cdots, a_{\ell}) \in \frak{t}_+^*(m) \subset \frak{t}^*(m)  \cong \R^\ell, \\
		\frak{b} := (b_1, \cdots, b_{\ell}) \in \frak{t}_+^*(m-1)  \subset \frak{t}^*(m-1) \cong \R^\ell, 
	\end{cases}
\]
\item \textbf{(Case 2) } when $m = 2 \ell + 2$, let
\[
	\begin{cases}
		\frak{a} = (a_1, \cdots, a_{\ell+1})\in \frak{t}_+^*(m) \subset \frak{t}^*(m) \cong \R^{\ell+1},  \\
		\frak{b} = (b_1, \cdots, b_{\ell})\in \frak{t}_+^*(m-1) \subset \frak{t}^*(m-1) \cong \R^\ell.
	\end{cases}
\]
\end{itemize}
We denote
\begin{itemize}
\item by $\mcal{O}^{(m)}_\frak{a}$ the co-adjoint $\mathrm{SO}(m)$-orbit of the block diagonal matrix $I^{(m)}_\frak{a}$ in Section~\ref{ssec_soNCases} and 
\item by $\mcal{O}^{(m-1)}_\frak{b}$ the co-adjoint $\rmso(m-1)$-orbit of the block diagonal matrix $I^{(m-1)}_\frak{b}$ in Section~\ref{ssec_soNCases}. 
\end{itemize}
Let
\begin{equation}\label{equ_intermediatestagetotal}
\mcal{O}^{\vphantom{(m-1)}}_{\frak{a}, \frak{b}} := \left\{ A \in \mcal{O}_\fa^{(m)} \mid A^{(m-1)} \in \mcal{O}^{(m-1)}_\frak{b} \right\}
\end{equation}
and define the projection map
\begin{equation}\label{equ_projerhoab}
		\rho^{\vphantom{(m-1)}}_{\fa, \fb} \colon  \mcal{O}^{\vphantom{(m-1)}} _{\frak{a}, \frak{b}} \rightarrow   \mcal{O}_{\frak{b}}^{(m-1)} \quad \left(A^{\vphantom{(m-1)}} \mapsto A^{(m-1)}\right)
\end{equation}
where $A^{(m-1)}$ is the $(m-1)$-th order leading principal submatrix of $A$. 
Also, the inverse image of the block diagonal matrix $I^{(m-1)}_{\frak{b}}$ under the projection~\eqref{equ_projerhoab} is denoted by $\widetilde{\mcal{O}}_{\frak{a}, \frak{b}}$, namely, 
\begin{equation}\label{equ_wildtildemcaloab}
\widetilde{\mcal{O}}^{\vphantom{(m-1)}}_{\frak{a}, \frak{b}} := \rho_{\fa, \fb}^{-1}(I^{(m-1)}_{\frak{b}}) \subset \mcal{O}_{\fa, \fb} \subset \mcal{O}_{\fa}^{(m)}
\end{equation}

Note that $\mcal{O}_{\frak{a}, \frak{b}}$ is non-empty if and only if 
\begin{equation}\label{equ_givenfrakab}
\begin{cases}
a_1 \geq b_1 \geq \cdots \geq a_\ell \geq |b_\ell | &\quad \mbox{when $ m = 2 \ell + 1$}\\
a_1 \geq b_1 \geq \cdots \geq a_\ell \geq b_\ell \geq |a_{\ell+1}| &\quad \mbox{when $ m = 2 \ell + 2$.}
\end{cases}
\end{equation}
by the min-max principle, see Figure~\ref{Fig_Wblockswith}.
The main proposition in this subsection is then as follows.

\begin{proposition}\label{prop_projbundle}
The projection $\rho^{\vphantom{(m-1)}}_{\fa, \fb}$ in~\eqref{equ_projerhoab} is an $\widetilde{\mcal{O}}^{\vphantom{(m-1)}}_{\frak{a}, \frak{b}}$-bundle over $\mcal{O}_{\frak{b}}^{(m-1)}$.
\end{proposition}

To prove Proposition~\ref{prop_projbundle}, 
recall that the co-adjoint $\rmso(m)$-action on $\mcal{O}_{\fa}^{(m)} \subset \frak{so}(m)^*$ is Hamiltonian where a moment map is given by 
the inclusion $\mcal{O}_\fa^{(m)} \hookrightarrow \frak{so}(m)^*$. Then $\rmso(m-1)$, as the subgroup of $\rmso(m)$ of the form
		\begin{equation}\label{equation_embedding}
			 \rmso(m-1) \cong 
			\begin{pmatrix}\rmso(m-1) & 0\\[0.3em]
                  	0 & 1\\
		     \end{pmatrix} \subset \rmso(m), 
		\end{equation}
acts on $\mcal{O}_{\fa}^{(m)}$ in a Hamiltonian fashion where the projection
\begin{equation}\label{equation_rhoasom-1}		
\rho_{\fa} \colon \mcal{O}_{\fa}^{(m)} \to \frak{so}(m-1)^* \quad \quad \left(A \mapsto A^{(m-1)}\right)
\end{equation}
becomes a moment map. 
Since $\rho_{\fa}$ is $\rmso(m-1)$-equivariant, we see that 
$\mcal{O}^{\vphantom{(m-1)}} _{\frak{a}, \frak{b}} = \rho_{\fa}^{-1}\left(\mcal{O}_{\fb}^{(m-1)}\right)$ inherits the $\rmso(m-1)$-action. 

\begin{lemma}[cf. \cite{GS_Th}]\label{lemma_transitive} 
	Let $\fa$ and $\fb$ be given in \eqref{equ_givenfrakab}. Then $\rmso(m-1)$ acts transitively on $\mcal{O}_{\fa, \fb}$. 
\end{lemma}

\begin{proof}
	If the Hamiltonian $G$-manifold $(M,\omega,\mu)$ admits a completely integrable system consisting of collective Hamiltonians, then the $G$-action is {\em multiplicity free} by Theorem~\ref{theorem_multiplicityfree}. 
	That is, if $\mu \colon M \rightarrow \frak{g}^*$ is the moment map for the action, then the induced 
	$G$-action on $\mu^{-1}(G\cdot \xi)$ is transitive for any $\xi \in \frak{g}^*$, see \cite[(2.4)]{GS_Th}. 
	
	We apply the above statement to $(\mcal{O}_{\fa}^{(m)}, \omega_{\fa}, \rho_{\fa})$. 
	Since the GC system on 
	$(\mcal{O}_{\fa}^{(m)}, \omega_{\fa})$ consists of collective Hamiltonians with respect to the $\rmso(m-1)$-action by Proposition~\ref{proposition_cisgc}, the action should be multiplicity free.
	Thus the action on $\mcal{O}^{\vphantom{(m-1)}} _{\frak{a}, \frak{b}} = \rho_{\fa}^{-1}\left(\mcal{O}_{\fb}^{(m-1)}\right)$ is transitive for every $\fb$.
\end{proof}

\begin{lemma}\label{lemma_equivariant_map}
The projection $\rho^{\vphantom{(m-1)}}_{\fa, \fb}$ in~\eqref{equ_projerhoab} is an $\mathrm{SO}(m-1)$-equivariant map. 
\end{lemma}

\begin{proof}
	Observe that $\rho_{\fa, \fb}$ is the restriction of $\rho_{\fa}$ in~\eqref{equation_rhoasom-1} onto $\mcal{O}^{\vphantom{(m-1)}} _{\frak{a}, \frak{b}}$.
	Since the moment map $\rho_{\fa}$ is $\rmso(m-1)$-equivariant, so is $\rho_{\fa, \fb}$.
\end{proof}

Now, we are ready to verify Proposition~\ref{prop_projbundle}. 

\begin{proof}[Proof of Proposition~\ref{prop_projbundle}]
By Lemma~\ref{lemma_transitive}, $\mathrm{SO}(m-1)$ acts transitively on $\mcal{O}_{\fa, \fb}$ so that $\mcal{O}_{\fa, \fb}$ is an orbit under the smooth $\mathrm{SO}(m-1)$-action in $\mcal{O}_{\frak{a}}^{(m)}$. 
Since $\mathrm{SO}(m-1)$ is compact, $\mcal{O}_{\fa, \fb}$ is an embedded smooth submanifold of $\mcal{O}_{\frak{a}}^{(m)}$. 
Hence, by Lemma~\ref{lemma_equivariant_map}, we have an $\rmso(m-1)$-equivariant map $\rho^{\vphantom{(m-1)}}_{\fa, \fb}$ between smooth manifolds. 
By the equivariant rank theorem (see \cite[Theorem 9.7]{Leesmooth} for instance), the transitivity of the $\mathrm{SO}(m-1)$-action and the surjectivity of $\rho^{\vphantom{(m-1)}}_{\fa, \fb}$ imply that $\rho^{\vphantom{(m-1)}}_{\fa, \fb}$ is submersive.
Finally, Ehresmann's fibration theorem yields that $\mcal{O}_{\fa, \fb}$ is a locally trivial fibration over $\mcal{O}^{(m-1)}_{\fb}$. 
Then Proposition~\ref{prop_projbundle} follows. 
\end{proof}

\begin{remark}
We remark that a similar argument was used by Lane \cite[Section 4.3]{Lane1}.
\end{remark}

\begin{corollary}\label{cor_fiberhomegenous}
Let $\scr{L}_{\frak{b}}$ be the stabilizer subgroup for the co-adjoint $\mathrm{SO}(m-1)$-action at $I^{(m-1)}_{\frak{b}}$. 
The space $\widetilde{\mcal{O}}^{\vphantom{(m-1)}}_{\frak{a}, \frak{b}}$ in~\eqref{equ_wildtildemcaloab} is a homogeneous $\scr{L}_{\frak{b}}$-space. 
\end{corollary}

\begin{proof}
It follows from Lemma~\ref{lemma_transitive} and~\ref{lemma_equivariant_map}.
\end{proof}

\subsection{Fibers at stages}\label{subsection_fibersatstages}

In this subsection, we describe the fiber of the bundle $\rho^{\vphantom{(m-1)}}_{\fa, \fb}$ in Proposition~\ref{prop_projbundle} and build up a connection between the fibers and the block combinatorics illustrated in Section~\ref{secTheTopologyOfGelfandCetlinFibers}. 
For a pair $(\frak{a}, \frak{b})$ in~\eqref{equ_givenfrakab}, consider the $W_m$-block in~\eqref{eq_wkunion}.
Setting 
\begin{equation}\label{equ_fafab||}
	|{\fa}| = (|{a}_1|, |{a}_2|, \cdots ) \quad \text{and} \quad |{\fb}| = (|{b}_1|, |{b}_2|, \cdots ),
\end{equation}
let $W_{\frak{a}, \frak{b}}$ be the $W_m$-shaped block with fillings by $|{\frak{a}}|$ and $|{\frak{b}}|$ as depicted in Figure~\ref{Fig_Wblockswith}. 
Note that, because of the pattern~\eqref{equ_givenfrakab}, all components of $|{\fa}|$ and $|{\fb}|$ are same as ${\fa}$ and ${\fb}$ except possibly for $a_{\ell + 1}$ when $m = 2 \ell + 2$ and for $b_{\ell}$ when $m = 2\ell + 1$.

\begin{figure}[h]
	\scalebox{0.6}{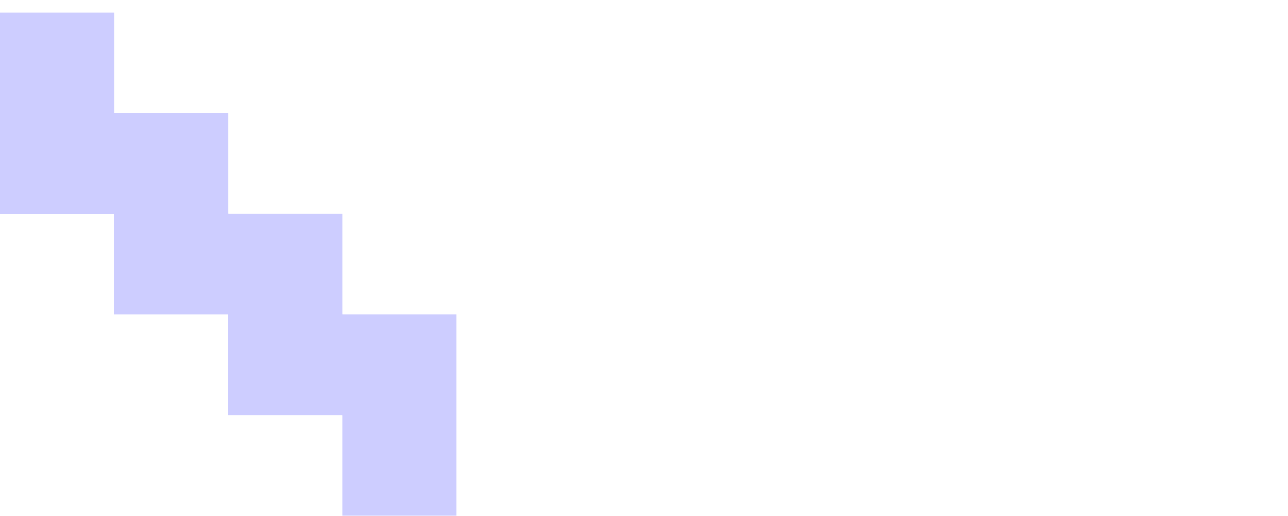}
	\caption{\label{Fig_Wblockswith} $W_{\frak{a}, \frak{b}}$-blocks.}	
\end{figure}

Recall the inequalities for the components of $(|\frak{a}|, |\frak{b}|)$ in~\eqref{equ_givenfrakab}.
We divide the block $W_m$ into several simple closed regions by cutting every shared edge of two unit boxes in $W_{\frak{a}, \frak{b}}$ labeled by either $a_j$ and $b_j$ with $|{a}_j| > |{b}_j|$ or ${b}_j$ and ${a}_{j+1}$ with $|{b}_j| > |{a}_{j+1}|$, see Figure~\ref{Fig_Wblockscut}. 
Regarding $W_{\frak{a}, \frak{b}}$ as the $W_m$-block, we then assign the topological space $S(\mcal{D})$ to each cut simple closed region $\mcal{D}$ of ${W}_m$ following the rule~\eqref{equ_fibersfromblocks}. 
Since $W_{\frak{a}, \frak{b}}$ is the $W_m$-block with fillings, a bottom point of $W_{\frak{a}, \frak{b}}$ makes sense. 
Perhaps, a clarification on contributions of $N$-blocks is necessary$\colon$
For an $N_\ell$-block $\mcal{D}$, we set 
$$
S(\mcal{D}) =
\begin{cases}
\mathbb{S}^{2\ell}  &\mbox{if $\mcal{D}$ contains zero components of $|{\frak{a}}|$ and $|{\frak{b}}|$.} \\
\textup{point} &\mbox{otherwise.}
\end{cases}
$$
As an example, consider $W_4$ in Figure~\ref{Fig_Fiberreading2} where $\frak{a} = (a,0)$ and $\frak{b} = (0)$. Since $a > 0$, $W_4$-block is divided into two blocks $\colon$ one unit block $\mcal{D}_1$ and one $N_1$-block $\mcal{D}_2$. $S(\mcal{D}_1)$ is a point and $S(\mcal{D}_2)$ is $\mathbb{S}^2$ since $\mcal{D}_2$ contains $(0,0)$.

We then put
\begin{equation}\label{equ_fibersatkthstageab}
S_{\frak{a}, \frak{b}} := \prod_\mcal{D} S (\mcal{D})
\end{equation}
where the product is taken over all divided simple closed regions of ${W}_{\frak{a}, \frak{b}}$. 

\begin{figure}[h]
	\scalebox{0.75}{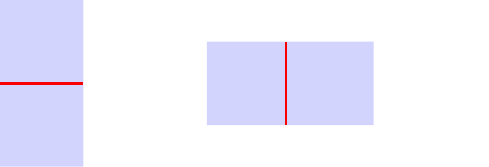}
	\vspace{-0.1cm}
	\caption{\label{Fig_Wblockscut} Cutting $W_{\frak{a}, \frak{b}}$-blocks.}	
\end{figure}

The following lemma characterizes the fibers at stages. The proof of Lemma~\ref{lemma_descriptionofsab} will be given in Section~\ref{secproofoflemma}.

\begin{lemma}\label{lemma_descriptionofsab}
The smooth manifold $\widetilde{\mcal{O}}_{\frak{a}, \frak{b}}$ in~\eqref{equ_wildtildemcaloab} is diffeomorphic to $\mcal{S}_{\frak{a}, \frak{b}}$ in \eqref{equ_fibersatkthstageab}.
\end{lemma}

Assuming Lemma~\ref{lemma_descriptionofsab}, we are ready to prove Proposition~\ref{proposition_mainiterated}.

\begin{proof}[Proof of Proposition~\ref{proposition_mainiterated}]
The point $\bf{u}$ determines the sequence $(\frak{a}^{(2)}, \cdots, \frak{a}^{(n)})$ of strings where $\nu(m) := \lfloor  \frac{m}{2} \rfloor$
\begin{equation}\label{equ_positionstrings}
\frak{a}^{(m)} := (u_{1,m-1}, \cdots, u_{\nu(m), m - \nu(m)}).
\end{equation}
The iterated bundle whose total space is the fiber $\left( \Phi_\lambda^{(n)} \right)^{-1}({\bf u})$ in \eqref{equ_iteratedbundlestr} is then given by the following pull-backed iterated bundle from Proposition \ref{prop_projbundle}$\colon$
\begin{equation}\label{figure_iterated_bundle}
	\xymatrix{
		 E_n = \overline{S_{n}}(f)  \ar[r] \ar[d]_{{q_{n}}}& \quad \cdots \quad \ar[r] \ar[d] & \iota^* \left( \mcal{O}_{\frak{a}^{(n)}, \frak{a}^{(n-1)}}\right)   \ar[r] \ar[d]^{\iota^* \rho_{\frak{a}^{(n)\vphantom{-1}}, \frak{a}^{(n-1)}}}  & \mcal{O}_{\frak{a}^{(n)}, \frak{a}^{(n-1)}} \ar@{^{(}->}[r]^{\iota} \ar[d]^{\rho_{\frak{a}^{(n)\vphantom{-1}}, \frak{a}^{(n-1)}}} & \mcal{O}^{(n)}_{\frak{a}^{(n)}}\\
		 E_{n-1} = \overline{S_{n-1}}(f) \ar[r] \ar[d]_{{q_{n-1}}} & \quad \cdots \quad \ar[r] \ar[d] & \mcal{O}_{\frak{a}^{(n-1)}, \frak{a}^{(n-2)}} \ar@{^{(}->}[r]^{\iota} \ar[d]^{\rho_{\frak{a}^{(n-1)}, \frak{a}^{(n-2)}}}   &  \mcal{O}_{\frak{a}^{(n-1)}}^{(n-1)}  & \\
		 \vdots  \ar[r] \ar[d] & \quad {\cdots} \quad   \ar@{^{(}->}[r] \ar[d] & \mcal{O}^{(n-2)}_{\frak{a}^{(n-2)}} &  &\\
		E_3 = \overline{S_{3}}(f) \ar@{^{(}->}[r]^{\iota} \ar[d]_{{q_3}} & \quad \cdots \quad &  & & \\
		E_2 = \overline{S_{2}}(f) = \{ \textup{pt} \} &   &  & &}
\end{equation}
where $\iota$'s are the canonical inclusions. 
By Lemma~\ref{lemma_descriptionofsab}, the fiber of $q_m$ is (diffeomorphic to) $\mcal{S}_{\frak{a}, \frak{b}}$.
Moreover, the dimension formula~\eqref{eq_dimformulaprop} is straightforward.
\end{proof}

\subsection{Isotropic property of Gelfand--Cetlin fibers}\label{subsectionisotropicpro}
One of the statements in Theorem~\ref{thm_fiber} claims that each GC fiber is isotropic in $(\mcal{O}_\lambda^{(n)}, \omega^{(n)}_\lambda)$. 
It will be verified in the subsection. 

Let $\omega_\fa^{(n)}$ be the Kirillov--Kostant--Souriau form on $\mcal{O}_\fa^{(n)}$.

\begin{lemma}\label{lemma_inherit}
	Let $p \in \mcal{O}_{\fa, \fb}$. For any vectors $\eta_1, \eta_2$ in $T_p \mcal{O}_{\fa, \fb}$, we have 
	\[
		\left(\omega_\fa^{(n)}\right)_p (\eta_1, \eta_2) = \left(\omega_\fb^{(n-1)}\right)_{\rho_{\fa, \fb}(p)} ( (\rho_{\fa, \fb})_* \eta_1, (\rho_{\fa, \fb})_* \eta_2 ).
	\]
\end{lemma}

\begin{proof}
	By Lemma \ref{lemma_transitive}, each $\eta_i$ can be expressed by the fundamental vector field generated by the $\rmso(n-1)$-action so that we have 
	$\eta_1 = \underline{X}_p$, $\eta_2 = \underline{Y}_p$
	for some $X, Y \in \frak{so}(n-1)$. Since the moment map for the $\rmso(n-1)$-action on $\mcal{O}_{\fa}^{(n)}$ is nothing but $\rho_{\fa, \fb}$ (see \eqref{equation_rhoasom-1}), 
	we apply Lemma \ref{lemma_preserve} to $M = \mcal{O}_{\fa}^{(n)}$ with $\mu = \rho_{\fa, \fb}$ and $\omega_{\mu(p)}^{\mathrm{KKS}} = \omega_{\fb}$ and we then obtain a proof.
\end{proof}

\begin{proposition}\label{prop_isotropic}
	Any GC fiber $\left( \Phi_\lambda^{(n)} \right)^{-1}({\bf u})$ is an isotropic submanifold of $(\mcal{O}_\lambda^{(n)}, \omega_\lambda^{(n)})$.
\end{proposition}

\begin{proof}	
	For any point ${\bf u} = ({u}_{i,j})\in \Delta_\lambda^{(n)}$, we have a family of non-increasing sequences such that 
	\begin{equation}\label{equ_u_sequence}
		\{ {\bf u}_{m+1} \}_{m = 1}^{n-1}, \quad {\bf u}_{m+1} := ({u}_{1, m}, {u}_{2, m-1}. \cdots) \in \R^{\lfloor \frac{m}{2} \rfloor}.
	\end{equation}
	In particular ${\bf u}_2 = ({u}_{1,1})$ and ${\bf u}_n = \lambda$ where $\lambda$ is in~\eqref{equ_sequoflambbda}. Note that $\mcal{O}^{(2)}_{\bf{u}_2} = \{ \mathrm{point} \}$.
	
	For any point $p$ in the fiber of $\Phi_\lambda^{(n)}$ over ${\bf u} \in \Delta_\lambda^{(n)}$, take any tangent vectors $\eta_1, \eta_2$ to the fiber. 
	Setting $p_0 := p$ and $\eta_{i}^{(0)} := \eta_{i}^{\vphantom{(0)}}$ for $i = 1,2$ and applying Lemma~\ref{lemma_inherit} successively, we obtain
\begin{align*}
(\omega_\lambda^{(n)})_p (\eta_1, \eta_2) &= (\omega_{{\bf u}_{n}, {\bf u}_{n-1}}^{(n)})_p (\eta^{(0)}_1, \eta^{(0)}_2) = (\omega_{{\bf u}_{n-1}, {\bf u}_{n-2}}^{(n-1)})_{p_1} (\eta_1^{(1)}, \eta_2^{(1)}) = \cdots = (\omega_{{\bf u}_3, {\bf u}_2}^{(3)})_{p_{n-1}} (\eta_1^{(n-2)}, \eta_2^{(n-2)})
\end{align*}
where
$p_k = \rho_{{\bf u}_{n-k},  {\bf u}_{n-k-1}} (p_{k-1})$ and  $\eta_i^{(k)} := (\rho_{{\bf u}_{n-i},  {\bf u}_{n-i-1}})_*(\eta_i^{(k-1)}).$
Then
$$
(\omega_{{\bf u}_3, {\bf u}_2}^{(3)})_{p_{n-1}} (\eta_1^{(n-2)}, \eta_2^{(n-2)})=0
$$
because $\rho_{{\bf u}_3, {\bf u}_2} \colon \mcal{O}_{{\bf u}_3, {\bf u}_2} \rightarrow \mcal{O}_{{\bf u}_2}^{(2)} = \{\mathrm{point}\}.$
	Therefore, $\omega_\lambda^{(n)}$ vanishes on ${(\Phi_\lambda^{(n)})^{-1}({\bf u})}$.	
\end{proof}


We close this section by generalizing Theorem~\ref{thm_fiber} on the classical Gelfand--Cetlin systems to the Gelfand--Cetlin systems on multiplicity free Hamiltonian $\rmso(n)$-manifolds (in the setting of \cite[Theorem 26]{Lane2}). 
		
	\begin{proposition}\label{theorem_general}
		Let $(M,\omega)$ be a Hamiltonian $\rmso(n)$-action with a proper moment map $\mu \colon M \rightarrow \frak{so}(n)^*$. Suppose that $(M,\omega)$ admits a Gelfand--Cetlin system $\Psi = \Phi \circ \mu \colon M \rightarrow \R^N$ in the sense of Definition~\ref{definition_GC_Multiplicity_free}.
		Then every fiber of $\Psi$ is an isotropic submanifold of $(M,\omega)$. Moreover, each fiber $\Psi^{-1}({\bf u})$ is the  
		total space of an iterated bundle
		\begin{equation}\label{equ_iterated_general}
			\Psi^{-1}({\bf{{u}}}) = E_{n+1} \stackrel{q_{n+1}} \longrightarrow \Phi^{-1}({\bf{{u}}})  = E_{n} \stackrel{q_{n}} \longrightarrow \cdots  
			\stackrel{q_4} \longrightarrow E_3  \stackrel{q_3} \longrightarrow E_2 = \mathrm{point}
		\end{equation}
		such that the fiber of each $q_i$ is 
		\begin{itemize}
			\item a point or a product of spheres if $ 3 \leq i \leq n$,
			\item the fiber $\mu^{-1}(x)$ of $x$ under the moment map $\mu$  where $x \in \Phi^{-1}({\bf u})$ if $i = n+1$.
		\end{itemize}
	\end{proposition}
	
	\begin{remark}\label{remark_comparelane2}
		In Section~\ref{secTheTopologyOfGelfandCetlinFibers}, we have described the fibers of $q_i$'s for $i$ with $3 \leq i \leq n$. 
		In the classical GC systems in~\eqref{equ_GCsystemsolambdabd}, the moment map $\mu$ can be simply taken as the inclusion $\mcal{O}_\lambda^{(n)} \to \frak{g}^*$. 
		Moreover, the inclusion $\mu$ induces the map $q_{n+1} \colon E_{n+1} \to E_n$ at the last stage of~\eqref{equ_iterated_general} and hence the fiber of $q_{n+1}$ is a single point. 
	\end{remark}

	\begin{proof}[Proof of Proposition \ref{theorem_general}]
		For any fixed ${\bf u} = (u_{i,j}) \in \R^N$, we have $\Psi^{-1}({\bf u}) = \mu^{-1}(\Phi^{-1}({\bf u}))$. 
		Note that $\Phi^{-1}({\bf u})$ is a GC fiber in $\mcal{O}_\lambda^{(n)}$ where $\lambda = (u_{1,n}, \cdots, u_{\lfloor n / 2 \rfloor, n+1 - \lfloor n / 2 \rfloor})$.
		In particular, $\Phi^{-1}({\bf u}) \subset \mcal{O}_\lambda$ so that 
		\[
			\mu^{-1}(\Phi^{-1}({\bf u})) \subset \mu^{-1}(\mcal{O}_\lambda).
		\]
		Thus, the description of the fibers of~\eqref{equ_iterated_general} follows from Theorem \ref{thm_fiber}.
		
		Since the action is multiplicity free by Theorem~\ref{theorem_multiplicityfree}, the induced $\rmso(n)$-action on $\mu^{-1}(\mcal{O}_\lambda)$ is transitive so that any tangent vector 	
		at any point $x \in \mu^{-1}(\Phi^{-1}({\bf u}))$ can be represented as the fundamental vector field $\underline{X}_x$ for some $X \in \frak{so}(n)$. 
		Applying Lemma \ref{lemma_preserve}, for any $\underline{X}_x, \underline{Y}_x \in T_x \mu^{-1}(\Phi^{-1}({\bf u}))$, we get
		\[
			\omega_x(\underline{X}_x, \underline{Y}_x) = \omega_\lambda (d\mu_x(\underline{X}_x), d\mu_x(\underline{Y}_x)) = 0 
		\]
		since $d\mu_x(\underline{X}_x)$ and $d\mu_x(\underline{Y}_x)$ are tangent to $\Phi^{-1}({\bf u})$, which is isotropic in $(\mcal{O}_\lambda, \omega_\lambda)$
		by Proposition \ref{prop_isotropic}.
	\end{proof}

\section{Fibers of iterated bundles at stages}\label{sectionproofoflemma}

In this section, we will complete the proof of Theorem~\ref{thm_fiber}. 

\subsection{Description of fibers of iterated bundles at stages}\label{secproofoflemma}

For a pair $(\frak{a}, \frak{b})$ in~\eqref{equ_givenfrakab}, consider a family of matrices of the form $Z_{\frak{a}, \frak{b}}({\bf{x},\bf{y}}) \in \mcal{O}_\frak{a}^{(m)}$
defined as follows. 

\begin{itemize}
\item \textbf{(Case 1)} When $m = 2 \ell + 1$ (odd), 
\begin{equation}\label{equ_maxtrix_odd}
Z_{\frak{a}, \frak{b}}(\bf{{x}},\bf{{y}}) := 
\begin{pmatrix}
   0 & b_1 & 0 & 0 & \cdots & 0  & 0 & x_1 \\
   -b_1 & 0 & 0 & 0 & \cdots & 0 & 0 & y_1 \\
   0 & 0 & 0 & b_2 & \cdots & 0 & 0 & x_2 \\
   0 & 0 & -b_2 & 0 & \cdots & 0 & 0 & y_2 \\
   \vdots & \vdots & \vdots & \vdots & \ddots & \vdots & \vdots & \vdots \\
   0 & 0 & 0 & 0 & \cdots & 0 & b_{\ell} & x_{\ell} \\   
   0 & 0 & 0 & 0 & \cdots & -b_{\ell} & 0 & y_{\ell} \\   
   -x_1 & -y_1 & -x_2 & -y_2 & \cdots & -x_{\ell} & -y_{\ell} & 0 \\   
\end{pmatrix}
\end{equation}
where $({\bf{{x}}},{\bf{{y}}}) \in \R^\ell \times \R^\ell$. 
\item \textbf{(Case 2)} When $m = 2 \ell + 2$ (even), 
\begin{equation}\label{equ_maxtrix_even}
Z_{\frak{a}, \frak{b}}(\bf{{x}},\bf{{y}}) := 
\begin{pmatrix}
   0 & b_1 & 0 & 0 & \cdots & 0  & 0 & 0 & x_1 \\
   -b_1 & 0 & 0 & 0 & \cdots & 0 & 0 & 0 & y_1 \\
   0 & 0 & 0 & b_2 & \cdots & 0 & 0 & 0 & x_2 \\
   0 & 0 & -b_2 & 0 & \cdots & 0 & 0 & 0 & y_2 \\
   \vdots & \vdots & \vdots & \vdots & \ddots & \vdots & \vdots & \vdots & \vdots \\
   0 & 0 & 0 & 0 & \cdots & 0 & b_{\ell} & 0 & x_{\ell} \\   
   0 & 0 & 0 & 0 & \cdots & -b_{\ell} & 0 & 0 & y_{\ell} \\   
   0 & 0 & 0 & 0 & \cdots & 0 & 0 & 0 & x_{\ell+1} \\   
   -x_1 & -y_1 & -x_2 & -y_2 & \cdots & -x_{\ell} & -y_{\ell} & -x_{\ell + 1} & 0 \\   
\end{pmatrix}
\end{equation}
where $({\bf{{x}}},{\bf{{y}}}) \in \R^{\ell+1} \times \R^\ell$. 
\end{itemize}

By~\eqref{equ_wildtildemcaloab}, $\widetilde{\mcal{O}}^{\vphantom{(m-1)}}_{\frak{a}, \frak{b}}$ is the subset of $\mcal{O}_\frak{a}^{(m)}$ consisting of 
$(m \times m)$-matrices of the form $Z^{\vphantom{(m-1)}}_{\frak{a}, \frak{b}}(\bf{{x}},\bf{{y}})$. Since an element in $\mcal{O}_\frak{a}^{(m)}$ can be characterized by its eigenvalues and its Pfaffian, we have the following lemma.

\begin{lemma} Let $\mcal{S}_{m}$ be the set of $(m \times m)$ skew-symmetric matrices. 
\begin{itemize}
\item When $m = 2 \ell + 1$ (odd), 
$$
\widetilde{\mcal{O}}_{\frak{a}, \frak{b}} = 
\left\{ A \in \mcal{S}_{m} \mid A = Z_{\frak{a}, \frak{b}}({\bf{{x}}},{\bf{{y}}}),\,\,  \textup{spec}(A) \equiv \langle \pm a_1 \sqrt{-1}, \dots, \pm a_\ell \sqrt{-1}, 0 \rangle \right\}.
$$
\item When $m = 2 \ell + 2$ (even),
$$
\widetilde{\mcal{O}}_{\frak{a}, \frak{b}} = 
\begin{cases}
\left\{ A \in \mcal{S}_{m} \mid A = Z_{\frak{a}, \frak{b}}({\bf{{x}}},{\bf{{y}}}),\,\,  \textup{spec}(A) \equiv \langle \pm a_1 \sqrt{-1}, \dots, \pm a_{\ell + 1} \sqrt{-1} \rangle \right\} &\mbox{if $a_{\ell+1} = 0$} \\
\left\{ A \in \mcal{S}_{m} \mid A = Z_{\frak{a}, \frak{b}}({\bf{{x}}},{\bf{{y}}}),\,\, \textup{pf}(A) > 0, \,\, \textup{spec}(A) \equiv \langle \pm a_1 \sqrt{-1}, \dots, \pm a_{\ell + 1} \sqrt{-1} \rangle \right\} &\mbox{if $a_{\ell+1} > 0$}\\
\left\{ A \in \mcal{S}_{m} \mid A = Z_{\frak{a}, \frak{b}}({\bf{{x}}},{\bf{{y}}}),\,\, \textup{pf}(A) < 0, \,\, \textup{spec}(A) \equiv \langle \pm a_1 \sqrt{-1}, \dots, \pm a_{\ell + 1} \sqrt{-1} \rangle \right\} &\mbox{if $a_{\ell+1} < 0$}.
\end{cases}
$$
\end{itemize}
\end{lemma}

\begin{proof}
It follows from~\eqref{equ_coadjointorbit2nu+1} and~\eqref{equ_coadjointorbit2nu}. 
\end{proof}

For later purpose, we compute characteristic polynomials of~\eqref{equ_maxtrix_odd} and~\eqref{equ_maxtrix_even}.

\begin{lemma}\label{lemma_char} The characteristic polynomial of $Z_{\frak{a}, \frak{b}}({\bf x}, {\bf y})$ can be written as 
\[
	\det \left( \xi I -  Z_{\frak{a}, \frak{b}}({\bf x},{\bf y}) \right) =  
	\begin{cases} \vspace{0.3cm}
		\xi \cdot \left( \prod_{i=1}^\ell (\xi^2 + b_i^2) + \sum_{j=1}^{\ell} \left( \frac{x_j^2 + y_j^2}{\xi^2 + b_j^2} \cdot \prod_{i=1}^\ell (\xi^2 + b_i^2) \right) \right) &\text{
		if $m = 2 \ell + 1$,}\\
		(\xi^2 + x_{\ell + 1}^2) \cdot \prod_{i=1}^\ell (\xi^2 + b_i^2) + \sum_{j=1}^{\ell} \left( \xi^2 \cdot \frac{x_j^2 + y_j^2}{\xi^2 + b_j^2} \cdot \prod_{i=1}^\ell (\xi^2 + b_i^2) \right)
		&\text{
		if $m = 2 \ell + 2$}.
	\end{cases}
\]
\end{lemma}

Consider the projection
\begin{equation}\label{equ_projectionoadr}
\Pi \colon \widetilde{\mcal{O}}_{\frak{a}, \frak{b}} \to 
\begin{cases}
\R^{\ell} \times \R^{\ell},  &\left(Z_{\frak{a}, \frak{b}}(\bf{{x}},\bf{{y}}) \mapsto (\bf{{x}},\bf{{y}})\right) \quad \mbox{ if $m = 2 \ell + 1$}, \\
\R^{\ell + 1} \times \R^{\ell},  &\left(Z_{\frak{a}, \frak{b}}(\bf{{x}},\bf{{y}}) \mapsto (\bf{{x}},\bf{{y}})\right) \quad \mbox{ if $m = 2 \ell + 2$}.
\end{cases}
\end{equation}
Since $\Pi$ is an embedding, it suffices to figure out the image of $\widetilde{\mcal{O}}_{\frak{a}, \frak{b}}$ under~\eqref{equ_projectionoadr} in order to describe the diffeomorphic type of $\widetilde{\mcal{O}}_{\frak{a}, \frak{b}}$. 
For the characterization, observe the following equivalences$\colon$
\begin{itemize}
\item $(m = 2 \ell + 1)$
\begin{equation}\label{equ_detpix}
\textup{spec}(Z_{\frak{a}, \frak{b}}({\bf{{x}}},{\bf{{y}}})) \equiv \langle \pm a_1 \sqrt{-1}, \dots, \pm a_{\ell} \sqrt{-1}, 0 \rangle \Leftrightarrow \det \left( \xi I -  Z_{\frak{a}, \frak{b}}({\bf x},{\bf y}) \right) = \xi \cdot \prod_{i=1}^\ell (\xi^2 + a_i^2)
\end{equation}
\item $(m = 2 \ell + 2)$
\begin{equation}\label{equ_detpix2}
\textup{spec}(Z_{\frak{a}, \frak{b}}({\bf{{x}}},{\bf{{y}}})) \equiv \langle \pm a_1 \sqrt{-1}, \dots, \pm a_{\ell+1} \sqrt{-1} \rangle \Leftrightarrow \det \left( \xi I -  Z_{\frak{a}, \frak{b}}({\bf x},{\bf y}) \right) = \prod_{i=1}^{\ell+1} (\xi^2 + a_i^2)
\end{equation}
\end{itemize}
In the case where $m = 2 \ell + 2$ and $a_{\ell + 1} \neq 0$, one additional condition on the Pfaffian should be imposed. A direct computation of the Pfaffian of $Z_{\frak{a}, \frak{b}}$ yields that
$\mathrm{pf}(Z_{\frak{a}, \frak{b}}({\bf{x}}, {\bf{y}})) = x_{\ell+1} \cdot \prod_{j=1}^\ell b_j.$
Because $a_{\ell + 1} \neq 0$ implies that $b_j > |a_{\ell + 1}| > 0$ for each $j$ with $j < \ell + 1$  by the min-max principle, we have
\begin{equation}\label{equ_detpix3}
\mathrm{pf}(Z_{\frak{a}, \frak{b}}({\bf{x}}, {\bf{y}})) > 0 \,\, (\mathrm{resp}. < 0) \Leftrightarrow x_{\ell + 1} > 0 \,\, (\mathrm{resp}. < 0).
\end{equation}
Indeed, $x_{\ell+1}$ is uniquely determined by $\frak{a}$ and $\frak{b}$ in this case. 

\begin{lemma}\label{lemma_Pfaffian}
When $m = 2 \ell + 2$, $x_{\ell+1}$ in~\eqref{equ_maxtrix_even} is determined as
\begin{equation}\label{equ_xell1}
x_{\ell+1} =  \frac{\prod_{j = 1}^{\ell + 1} a_i}{\prod_{j=1}^\ell b_j}
\end{equation}
provided that ${\Pi_{j=1}^\ell b_j} \neq 0$ 
\end{lemma}

\begin{proof}
By comparing 
$\mathrm{pf}(Z_{\frak{a}, \frak{b}}({\bf{x}}, {\bf{y}})) = x_{\ell+1} \cdot \prod_{j=1}^\ell b_j$ and
$\mathrm{pf}(I^{(m)}_{\frak{a}}) = {\prod_{j = 1}^{\ell + 1} a_i}$,
\eqref{equ_xell1} follows from $\mathrm{pf}(I^{(m)}_{\frak{a}}) = \mathrm{pf}(Q^T Z_{\frak{a}, \frak{b}}({\bf x}, {\bf y}) Q) = \det(Q) \cdot 
\mathrm{pf}(Z_{\frak{a}, \frak{b}}({\bf x}, {\bf y}))$ for some $Q \in \mathrm{SO}(2\ell + 2)$.
\end{proof}

Recall from Corollary~\ref{cor_fiberhomegenous} that $\widetilde{\mcal{O}}_{\frak{a}, \frak{b}}$ is a homogeneous $\scr{L}_{\frak{b}}$-space where $\scr{L}_{\frak{b}}$ is the stabilizer subgroup for the co-adjoint $\mathrm{SO}(m-1)$-action at $I^{(m-1)}_{\frak{b}}$. 
The stabilizer subgroup $\scr{L}_{\frak{b}}$ as a subgroup of $\mathrm{SO}(m-1)$ acts on $\R^{m-1}$ linearly.
It is straightforward to see that the embedding $\Pi$ in~\eqref{equ_projectionoadr} is $\scr{L}_{\frak{b}}$-equivariant. 
Let $\scr{S}_{\frak{a},\frak{b}}$ be the stabilizer subgroup of a point in the image $\Pi( \widetilde{\mcal{O}}_{\frak{a}, \frak{b}})$ for the action of $\scr{L}_{\frak{b}}$.
We then obtain the following description for $\widetilde{\mcal{O}}_{\frak{a}, \frak{b}}$.

\begin{lemma}\label{lemmaOabls}
The space $\widetilde{\mcal{O}}_{\frak{a}, \frak{b}}$ is diffeomorphic to $\scr{L}_{\frak{b}} / \scr{S}_{\frak{a},\frak{b}}$.
\end{lemma}

Thus, by analyzing $\scr{L}_{\frak{b}}$ and $\scr{S}_{\frak{a},\frak{b}}$, we may describe the fiber $\widetilde{\mcal{O}}_{\frak{a}, \frak{b}}$. 
We begin by describing $\scr{L}_{\frak{b}}$. 

\begin{lemma}\label{lemma_levisubgroup}
Suppose that $\frak{b} = (b_1, \cdots, b_\ell)$ satisfies
$$
b_1 = \dots = b_{\ell_1}  >  b_{\ell_1 + 1} = \dots = b_{\ell_2} > b_{\ell_2 + 1} = \dots > b_{\ell_{r-1} + 1}  = \dots = 
b_{\ell_r - 1}  = 
|b_{\ell_r}| \geq 0
$$
where $\ell_r = \ell$. Set $\ell_0 = 0$. 
Then $\scr{L}_\frak{b}$ is diffeomorphic to the product $\prod_{j=1}^r \scr{L}_j$ where $\scr{L}_j$ is determined as follows.
\begin{itemize}
\item When $m = 2 \ell + 1$ (odd),  
$$
\scr{L}_j := \mathrm{SO}(2\ell_j -  2\ell_{j-1}) \quad \mbox{for $j = 1, \cdots, r.$}
$$  
\item When $m = 2 \ell + 2$ (even), 
\begin{align*}
&\scr{L}_j := \mathrm{SO}(2\ell_j -  2\ell_{j-1}) \,\, \quad \quad \quad \quad \quad  \quad \quad \quad \quad \quad \quad \quad \quad \quad \mbox{for $j = 1, \cdots, r-1$}, \\
&\scr{L}_{r} :=
\begin{cases}
\mathrm{SO}(2\ell_r -  2\ell_{r-1}) \times \mathrm{SO}(1) \simeq \mathrm{SO}(2\ell_r -  2\ell_{r-1}) &\mbox{if $b_{\ell_r} \neq 0$}, \\
\mathrm{SO}(2\ell_r -  2\ell_{r-1} + 1)  &\mbox{if $b_{\ell_r} =0$}.
\end{cases}
\end{align*}
\end{itemize}
\end{lemma}

To analyze $\scr{S}_{\frak{a},\frak{b}}$, observe that the only groups that can occur as stabilizers for the action of $\scr{L}_\frak{b}$ on $\R^{m-1}$ are 
\begin{equation}\label{equ_stabilizergroup}
\scr{S}_{\frak{a},\frak{b}} \simeq \prod_{j=1}^r \scr{S}_j
\end{equation}
where $\scr{S}_j$ is diffeomorphic to either $\mathrm{SO}(k)$ or $\mathrm{SO}(k-1)$ provided $\scr{L}_j = \mathrm{SO}(k)$.
Notice that
\begin{itemize}
\item
$\scr{S}_j \simeq \mathrm{SO}(k)$ if all components of $({\bf{x}}, {\bf{y}})$ corresponding to $\scr{L}_j$ in the image $\Pi( \widetilde{\mcal{O}})$ vanish.
\item
$\scr{S}_j \simeq \mathrm{SO}(k-1)$ otherwise.
\end{itemize}

We first deal with the most generic case. 

\begin{lemma}\label{lemma_case0}
Suppose that all elements of $|\frak{{a}}| \cup |\frak{{b}}|$ in~\eqref{equ_givenfrakab} are pairwise distinct and non-zero. 
Let $({\bf x}, {\bf y})  = \Pi (Z_{\fa, \fb}({\bf x}, {\bf y})) \in \R^{m-1}$ and $\scr{S}_{\frak{a},\frak{b}}$ the stabilizer subgroup of $({\bf x}, {\bf y})$.
Then each factor $\scr{S}_j$ in~\eqref{equ_stabilizergroup} is isomorphic to $\mathrm{SO}(1)$, i.e., trivial.
In particular, $\widetilde{\mcal{O}}_{\frak{a}, \frak{b}}$ is diffeomorphic to $(\mathrm{SO}(2)/\mathrm{SO}(1))^\ell \simeq (\mathbb{S}^1)^\ell \simeq
\mcal{S}_{\frak{a}, \frak{b}}$ in ~\eqref{equ_fibersatkthstageab}.
\end{lemma}

\begin{proof}
By Lemma~\ref{lemma_levisubgroup}, $\scr{L}_j \simeq \mathrm{SO}(2)$ for $j = 1, \dots, r = \ell$. 
So, it remains to show that $\scr{S}_j = \mathrm{SO}(1)$ for $j = 1, \dots, \ell$. 
The characteristic polynomial $\det(\xi I - Z_{(\fa,\fb)}({\bf x}, {\bf y}))$ in Lemma \ref{lemma_char} is written as
\begin{equation}\label{equ_det_even}	
	D(\xi) := \begin{cases} \vspace{0.2cm}
		\xi \left( \prod_{i=1}^\ell (\xi^2 + b_i^2) + \sum_{j=1}^\ell \left( \frac{x_j^2 + y_j^2}{\xi^2 + b^2_j} \cdot \prod_{i=1}^\ell (\xi^2 + b^2_i) \right) \right) &\mbox{if $m = 2 \ell + 1$,}\\
		(\xi^2 + x^2_{\ell+1}) \cdot \prod_{i=1}^\ell (\xi^2 + b^2_i) + \sum_{j=1}^\ell \xi^2 \left( \frac{x_j^2+y_j^2}{\xi^2 + b^2_j} \cdot \prod_{i=1}^\ell (\xi^2 + b^2_i) \right) \quad &\mbox{if $m = 2 \ell + 2$}.
	\end{cases}
\end{equation}
Lemma~\ref{lemma_Pfaffian} determines $x_{\ell+1}$ when $m = 2 \ell + 2$ and $a_{\ell + 1} \neq 0$.
Since $Z_{\fa, \fb}({\bf{x}}, {\bf{y}}) \in \mcal{O}_\fa^{(m)}$,~\eqref{equ_detpix} and~\eqref{equ_detpix2} yield
\begin{equation}\label{equ_det_eventwo}
D(\xi) = 
\begin{cases}	
		\xi \cdot \prod_{i=1}^{\ell} (\xi^2 + a_i^2)   \quad &\mbox{if $m = 2 \ell + 1$}\\
		\prod_{i=1}^{\ell+1} (\xi^2 + a_i^2) \quad &\mbox{if $m = 2 \ell + 2$}.
\end{cases}
\end{equation}
By plugging $\xi^2 = -b^2_i$ for $i = 1, \cdots, \ell$ into~\eqref{equ_det_even} and~\eqref{equ_det_eventwo}, we obtain that $(x_i^2 + y_i^2)$ ($i = 1, \cdots, \ell$) is strictly positive. Therefore, at least one component of $(x_i, y_i)$ must not vanish, which yields $\scr{S}_i = \mathrm{SO}(1)$ for $i = 1, \dots, \ell$. Thus, $\widetilde{\mcal{O}}_{\frak{a}, \frak{b}} \simeq (\mathrm{SO}(2)/\mathrm{SO}(1))^\ell \simeq (\mathbb{S}^1)^\ell$.
Also, observe that $\mcal{S}_{\frak{a}, \frak{b}}$ arises from $\ell$ many unit blocks (containing bottom points of the $W$-block) and hence $\mcal{S}_{\frak{a}, \frak{b}} \simeq (\mathbb{S}^1)^\ell$.
\end{proof}

We turn to the remaining non-generic cases where the elements of $|{\fa}| \cup |{\fb}|$ are \emph{not} pairwise distinct. 
There exists a subsequence in $|{\fa}| \cup |{\fb}|$ whose elements are all equal. We call such a 
subsequence a {\em constant string} in $|{\fa}| \cup |{\fb}|$. 

In \textbf{(Case 1)}, there are eight possible types of constant strings in $\fa \cup \fb$ as follows.
\begin{equation}\label{equ_pattern_odd}
\begin{cases}
(1) \,\, a_s > b_s = a_{s+1} = \cdots = a_\ell = |b_\ell| = 0,  &(1') \,\, b_s > a_{s+1} = \cdots = a_\ell = |b_\ell| = 0, \\
(2) \,\, a_s > b_s = a_{s+1} = \cdots = a_\ell = |b_\ell| > 0,  &(2') \,\, b_s > a_{s+1} = \cdots = a_\ell = |b_\ell| > 0, \\
(3) \,\, a_s > b_s = a_{s+1} = \cdots = a_{s + t} > |b_{s+t}|,  &(3') \,\, b_s > a_{s+1} = \cdots = a_{s + t} > |b_{s+t}|, \\
(4) \,\, a_s > b_s = a_{s+1} = \cdots = b_{s+t} > a_{s + t + 1}, &(4') \,\, b_s > a_{s+1} = \cdots  = b_{s+t} > a_{s + t + 1}, 
\end{cases}
\end{equation}

\begin{lemma}\label{lemma_case_odd} Assume that $m = 2\ell + 1$ and 
let $(x_1, y_1, \cdots, x_\ell, y_\ell)  = \Pi (Z_{\fa, \fb}({\bf x}, {\bf y})) \in \R^{m-1}$ in \eqref{equ_maxtrix_odd}. 
Let $\scr{S}_{\frak{a},\frak{b}}$ be the stabilizer subgroup of $\Pi(Z_{\fa, \fb}({\bf x}, {\bf y}))$.
\begin{enumerate}
	\item Assume that $|{\fa}| \cup |{\fb}|$ contains a constant string of Type $(1)$, $(2)$, or $(4)$ in~\eqref{equ_pattern_odd}. 
	Suppose that the factor $\scr{L}_\tau$ of $\scr{L}_{\frak b}$ corresponding to $b_\bullet$'s in the constant string in Lemma~\ref{lemma_levisubgroup} is $\mathrm{SO}(k)$. 
	Then the factor $\scr{S}_\tau$ of $\scr{S}_{\frak a, \frak b}$ in~\eqref{equ_stabilizergroup} is $\mathrm{SO}(k-1)$ (up to an isomorphism). 	
	\item Assume that $|{\fa}| \cup |{\fb}|$ contains a constant string of the other types in~\eqref{equ_pattern_odd}. 
	Suppose that the factor $\scr{L}_\tau$ of $\scr{L}_{\frak b}$ corresponding to $b_\bullet$'s in the constant string in Lemma~\ref{lemma_levisubgroup} is $\mathrm{SO}(k)$. 
	Then the factor $\scr{S}_\tau$ of $\scr{S}_{\frak a, \frak b}$ in~\eqref{equ_stabilizergroup} is $\mathrm{SO}(k)$. 	
\end{enumerate}
In particular, $\scr{L}_\tau / \scr{S}_\tau$ is diffeomorphic to ${S}(\mcal{D})$ in~\eqref{equ_fibersatkthstageab} where $\mcal{D}$ is the block containing the corresponding constant string.
\end{lemma}

\begin{proof}[Sketch of Proof]
In this case, by Lemma~\ref{lemma_char}, we have
	\begin{equation}\label{equ_det_odd_re}
		\displaystyle \prod_{i=1}^{\ell} (\xi^2 + a^2_i)	=  \displaystyle \prod_{i=1}^\ell (\xi^2 + b^2_i) + \sum_{j=1}^\ell \left( \frac{x_j^2 + y_j^2}{\xi^2 + b^2_j} \cdot \displaystyle \prod_{i=1}^\ell (\xi^2 + b^2_i) \right).
	\end{equation}
	
	We shall prove 	Lemma~\ref{lemma_case_odd} (1) for Type $({1})$ (where the other types (2) and (4) can be similarly proven).
Observe that (LHS) and (RHS) of \eqref{equ_det_odd_re} have a common factor $(\xi^2  + a^2_\ell)^{\ell - s} = (\xi^2 +b^2_s)^{\ell - s}$.
Dividing both sides by $(\xi^2 + b^2_s)^{\ell - s}$, we have
	\begin{equation}\label{equ_det_odd_re_1}
			\displaystyle \prod_{i=1}^{s} (\xi^2 + a^2_i) = \displaystyle \prod_{i=1}^{s} (\xi^2 + b^2_i) + \sum_{j=1}^\ell \left( \frac{x_j^2 + y_j^2}{\xi^2 + b^2_j} \cdot \displaystyle \prod_{i=1}^{s} (\xi^2 + b^2_i) \right). 
	\end{equation}
	Therefore, by plugging $\xi^2 = - b^2_s$, we have
	\[
		\prod_{i=1}^s (a^2_i - b_s^2) = \sum_{j=s}^\ell (x_j^2 +y_j^2 ) \cdot \prod_{i=1}^{s-1} (b^2_i - b_s^2) 
	\]
	where $a_i^2 - b_s^2 > 0$ for every $i (\leq s)$ and $b_i^2 - b_s^2 > 0$ for every $i (< s)$.
	Consequently, any element $(x_1, y_1, \cdots, x_\ell, y_\ell) \in \R^{m-1}$ satisfies $\sum_{j=s}^\ell (x_j^2 +y_j^2) > 0$. Thus, Lemma~\ref{lemma_case_odd} (1)  for the Type $({1})$ follows.
	Moreover, $\scr{L}_\tau / \scr{S}_\tau \simeq \mathbb{S}^{2\ell - 2s + 1} \simeq S(\mcal{D})$ where $\mcal{D}$ is the $M_{\ell - s}$-block.
	
For Type $({1}')$, (LHS) of~\eqref{equ_det_odd_re} has a common factor $(\xi^2 + a^2_{s+1})^{\ell - s} = (\xi^2 + b^2_{s+1})^{\ell - s}$, 
while (RHS) of \eqref{equ_det_odd_re} has a common factor $(\xi^2 + b^2_{s+1})^{\ell - s - 1}$.
Dividing both sides by $(\xi^2 + b^2_{s+1})^{\ell - s - 1}$ and plugging in $\xi^2 = - b^2_{s+1}$ into
	\begin{equation}\label{equ_det_odd_re_2}
			\displaystyle \prod_{i=1}^{s+1} (\xi^2 + a^2_i) = \displaystyle \prod_{i=1}^{s+1} ( \xi^2 + b^2_i) + \sum_{j=1}^\ell \left( \frac{x_j^2 + y_j^2}{\xi^2 + b^2_j} \cdot \displaystyle \prod_{i=1}^{s+1} (\xi^2 + b^2_i) \right), 
	\end{equation}
we derive
	\[
		0 = \sum_{j=s+1}^\ell (x_j^2 + y_j^2 ) \cdot \prod_{i=1}^{s} (b^2_i - b^2_{s+1}),  
	\]
	which implies that $\sum_{j=s+1}^\ell (x_j^2 + y_j^2) = 0$.	Moreover, $\scr{L}_\tau / \scr{S}_\tau \simeq \{ \textup{pt} \} \simeq S(\mcal{D})$ where $\mcal{D}$ is the block containing the constant string of Type $({1}')$.
We omit the other cases because the proof is similar.	
\end{proof}   	   

Similarly, in \textbf{(Case 2)}, there are eight possible types of constant strings in $\fa \cup \fb$ as follows.
\begin{equation}\label{equ_pattern_even}
\begin{cases}
(5) \,\, a_s > b_s = a_{s+1} = \cdots = |a_{\ell + 1}| = 0,  &(5') \,\, b_s > a_{s+1} = b_{s+1} = \cdots = |a_{\ell + 1}| = 0,\\
(6) \,\, a_s > b_s = a_{s+1} = \cdots = |a_{\ell + 1}| > 0,  &(6') \,\, b_s > a_{s+1} = b_{s+1} = \cdots = |a_{\ell + 1}| > 0, \\
(7) \,\, a_s > b_s = a_{s+1} = \cdots = a_{s + t} > b_{s+t}, &(7') \,\, b_s > a_{s+1} = \cdots = a_{s + t} > b_{s+t}, \\
(8) \,\, a_s > b_s = a_{s+1} = \cdots = b_{s+t} > |a_{s + t + 1}|,  &(8') \,\, b_s > a_{s+1} = \cdots = b_{s+t} > |a_{s + t + 1}| 
\end{cases}
\end{equation}

\begin{lemma}\label{lemma_case_even} Assume that $n = 2\ell + 2$ and 
let $(x_1, y_1, \cdots, x_\ell, y_\ell, x_{\ell+1})  = \Pi(Z_{\fa, \fb}({\bf x}, {\bf y})) \in \R^{m-1} \}$ in \eqref{equ_maxtrix_even}. 
Let $\scr{S}_{\frak{a},\frak{b}}$ be the stabilizer subgroup of $\Pi(Z_{\fa, \fb}({\bf x}, {\bf y}))$.
\begin{enumerate}
	\item Suppose that $|{\fa}| \cup |{\fb}|$ contains a constant string of Type $(5)$ in~\eqref{equ_pattern_even}. 
	The factor $\scr{L}_\tau$ of $\scr{L}_{\frak b}$ corresponding to $b_\bullet$'s in this constant string is $\mathrm{SO}(2\ell - 2s + 3)$. 
	Then the factor $\scr{S}_\tau$ of $\scr{S}_{\frak a, \frak b}$ in~\eqref{equ_stabilizergroup} is $\mathrm{SO}(2\ell - 2s +2)$  (up to an isomorphism). 	
	\item Suppose that $|{\fa}| \cup |{\fb}|$ contains a constant string of Type $(8)$ in~\eqref{equ_pattern_even}. 
	The factor $\scr{L}_\tau$ of $\scr{L}_{\frak b}$ corresponding to $b_\bullet$'s in this constant string is $\mathrm{SO}(2t + 2)$. 
	Then the factor $\scr{S}_\tau$ of $\scr{S}_{\frak a, \frak b}$ in~\eqref{equ_stabilizergroup} is $\mathrm{SO}(2t+1)$  (up to an isomorphism). 
	\item Assume that $\overline{\fa} \cup \overline{\fb}$ contains a constant string of the other types as in \eqref{equ_maxtrix_odd}.
	Suppose that the factor $\scr{L}_\tau$ of $\scr{L}_{\frak b}$ corresponding to $b_\bullet$'s in the constant string in Lemma~\ref{lemma_levisubgroup} is $\mathrm{SO}(k)$. 
	Then the factor $\scr{S}_\tau$ of $\scr{S}_{\frak a, \frak b}$ in~\eqref{equ_stabilizergroup} is $\mathrm{SO}(k)$. 		
\end{enumerate}
In particular, $\scr{L}_\tau / \scr{S}_\tau$ is diffeomorphic to ${S}(\mcal{D})$ in~\eqref{equ_fibersatkthstageab} where $\mcal{D}$ is the block containing the corresponding constant string.
\end{lemma}

\begin{proof}[Sketch of Proof]
	We only give the proof for type $(5)$ and $(6)$. (The other types can be similarly proven.)
In this case, 
	\begin{equation}\label{equ_det_even_re}
		\prod_{i=1}^{\ell+1} (\xi^2 + a^2_i)  = (\xi^2 + x_{\ell+1}^2) \cdot \prod_{i=1}^\ell (\xi^2 + b^2_i) + \sum_{j=1}^\ell \xi^2 \left( \frac{x^2_j + y_j^2}{\xi^2 + b^2_j} \cdot \prod_{i=1}^\ell (\xi^2 + b^2_i) \right) 
	\end{equation}
	In type $(5)$, notice that Lemma~\ref{lemma_Pfaffian} can \emph{not} be applied because $b_\ell = 0$. 
	Since (LHS) and (RHS) of~\eqref{equ_det_even_re} 
	have a common factor $\xi^{2\ell - 2s + 2}$. Dividing both sides of \eqref{equ_det_even_re} by $\xi^{2\ell - 2s + 2}$, we get 
	\begin{equation}\label{equ_det_even_re_re}
		\prod_{i=1}^{s} (\xi^2 + a^2_i) = (\xi^2 + x_{\ell+1}^2) \cdot \prod_{i=1}^{s-1} (\xi^2 + b^2_i) + \sum_{j=1}^\ell \xi^2 \left( \frac{x_j^2 + y_j^2}{\xi^2 + b^2_j} \cdot \prod_{i=1}^{s-1} (\xi^2 + b^2_i) \right). 
	\end{equation}	
	Letting $\xi = 0$, we obtain
	$$
		\sum_{j=s}^\ell (x_j^2 + y_j^2) + x_{\ell+1}^2 = \frac{\prod_{i=1}^s a^2_i}{\prod_{i=1}^{s-1} b^2_i} > 0
	$$
	as desired.
	It yields that $\scr{S}_\tau \simeq \mathrm{SO}(2\ell - 2s +2)$. 
	Therefore, $\scr{L}_\tau / \scr{S}_\tau \simeq \mathbb{S}^{2\ell - 2s + 2} \simeq S(\mcal{D})$ where $\mcal{D}$ is the $N_{\ell - s +1}$-block.

 	In Type $(6)$, applying Lemma~\ref{lemma_Pfaffian}, $x_{\ell+1}$ is determined uniquely. 
Observe that (LHS) and (RHS) of~\eqref{equ_det_even_re} have a common factor $(\xi^2 + a^2_{s+1})^{\ell - s} = (\xi^2 + b^2_{s})^{\ell - s}$. Plugging $\xi^2 = - a^2_{s+1} = - b^2_s$ into
	\begin{equation}\label{equ_det_even_re_re6}
		\prod_{i=1}^{s+1} (\xi^2 + a^2_i) = (\xi^2 + x_{\ell+1}^2) \cdot \prod_{i=1}^{s} (\xi^2 + b^2_i) + \sum_{j=1}^\ell \xi^2 \left( \frac{x_j^2 + y_j^2}{\xi^2 + b^2_j} \cdot \prod_{i=1}^{s} (\xi^2 + b^2_i) \right),
	\end{equation}	
	we deduce
	$$
		0 = b^2_s \cdot  \prod_{i=1}^{s-1} (b^2_i - b^2_{s}) \cdot \sum_{j=s}^\ell (x_j^2 + y_j^2)  
	$$
	and therefore $\sum_{j=s}^\ell (x_j^2 + y_j^2) = 0$. Thus, $\scr{S}_\tau \simeq \mathrm{SO}(2\ell - 2s +3)$.
	Thus, $\scr{L}_\tau / \scr{S}_\tau \simeq \{ \textup{pt} \} \simeq S(\mcal{D})$ where $\mcal{D}$ is the block containing the constant string Type $(6)$.
	Even though $\mcal{D}$ is an $N$-block, it is \emph{above} the base line and hence $S(\mcal{D})$ is a point. 
\end{proof}

Now we are ready to prove Lemma~\ref{lemma_descriptionofsab} for the case where $(\frak{a}, \frak{b})$ contains constant strings. 

\begin{proof}[Proof of Lemma~\ref{lemma_descriptionofsab}]
By Lemma~\ref{lemmaOabls}, $\widetilde{\mcal{O}}_{\frak{a}, \frak{b}}$ is diffeomorphic to $\scr{L}_{\frak{b}} / \scr{S}_{\frak{a},\frak{b}}$.  
Moreover, Lemma~\ref{lemma_case0},~\ref{lemma_case_odd},~\ref{lemma_case_even}, and~\eqref{equ_fibersatkthstageab} yield that 
\begin{equation}\label{equation_prod_sphere}
	\widetilde{\mcal{O}}_{\frak{a}, \frak{b}} \simeq \prod_{j=1}^r (\scr{L}_j / \scr{S}_j) \simeq \prod_\mcal{D} S (\mcal{D}) \simeq S_{\frak{a}, \frak{b}}.
\end{equation}
In sum, $\widetilde{\mcal{O}}_{\frak{a}, \frak{b}}$ is diffeomorphic to $S_{\frak{a}, \frak{b}}$. 
\end{proof}

\subsection{Diffeomorphism types of Gelfand--Cetlin fibers}\label{sec_gcfiberdiffeo}

To finish the proof of Theorem~\ref{thm_fiber}, it remains to prove the following statement. 

\begin{proposition}\label{proposition_twofibersarediffeo}
For any two points $\bf{u}$ and $\bf{u}^\prime$ in the relative interior of a single face $f$, the GC fibers over $\bf{u}$ and $\bf{u}^\prime$ are diffeomorphic.
\end{proposition}

The positions of $\bf{u}$ and $\bf{u}^\prime$ determine two sequences 
$(\fa^{(2)}, \dots, \fa^{(n)})$ and $(\fa^{\prime \, (2)}, \dots, \fa^{\prime \, (n)})$
of strings as in~\eqref{equ_positionstrings}. 
For notational simplicity, we put $\frak{a}^{\vphantom{\prime}} := \frak{a}^{(m)}$ and $\frak{a}^\prime := \frak{a}^{\prime \, (m)}$ ($2 \leq m \leq n$). 
We can define an $\rmso(m)$-equivariant diffeomorphism $\eta^{(m)} \colon \mcal{O}_{\fa^{\vphantom{\prime}}}^{(m)} \to \mcal{O}_{\fa^\prime}^{(m)}$ such that $I_{\fa^{\vphantom{\prime}}}^{(m)} \mapsto I_{\fa^\prime}^{(m)}$.
It is well-defined since $\mathrm{SO}(m)$ acts transitively on both $\mcal{O}^{(m)}_{\frak{a}^{\vphantom{\prime}}}$ and $\mcal{O}^{(m)}_{\frak{a}^\prime}$ (by conjugation)
and the stabilizer of $I_{\fa}^{(m)} \in \mcal{O}_{\fa}^{(m)}$ is same as that of $I_{\fa'}^{(m)} \in \mcal{O}_{\fa'}^{(m)}$. 
More generally we have the following.

\begin{lemma}\label{lemma_fa}
	Let $\fa_{t} := (1-t)\fa + t \fa^\prime$ for $t \in [0,1]$, which parametrizes the line segment from $\fa$ to $\fa^\prime$.
	Then there exists a parameter family of $\rmso(m)$-equivariant diffeomorphisms
	\begin{equation}\label{equation_eta}
		\eta^{m}_{t} \colon \mcal{O}_{\fa_{\vphantom{t}}}^{(m)} \rightarrow \mcal{O}_{\fa_{t}}^{(m)}
	\end{equation}
	such that $\eta^{m}_{t}(I_{\fa}^{(m)}) = I_{\fa_{t}}^{(m)}$.
\end{lemma}

\begin{proof}
	Since the linear interpolation $\{\fa_{t}\}$ of the points $\bf{u}^{\vphantom{\prime}}$ and $\bf{u}^\prime$ lies in the relative interior of the same face $f$, the inequality pattern does \emph{not} change as $t \in [0,1]$ changes. 
	In particular, the stabilizer of $I_{\fa_t}^{(m)}$ coincides with that of $I_{\fa}^{(m)}$. 
	It finishes the proof.
\end{proof}

For a fixed $m \geq 3$, let $(\frak{a}, \frak{b}) = (\fa^{(m)}, \fa^{(m-1)})$ and $(\frak{a}^\prime, \frak{b}^\prime) = (\fa^{\prime\, (m)}, \fa^{\prime\, (m-1)})$. 
We construct an $\rmso(m-1)$-equivariant diffeomorphism $\eta_{m-1}^{m} \colon \mcal{O}_{\fa^{\vphantom{\prime}}, \fb^{\vphantom{\prime}}} \rightarrow \mcal{O}_{\fa^\prime, \fb^\prime}$
as follows. By our choice of ${\bf u}$ and ${\bf u}'$, the pairs of strings $(\fa, \fb)$ and $(\fa', \fb')$  have the same inequality pattern. 
Also note that $\rmso(m-1)$, as a subgroup of $\rmso(m)$ 
(with respect to the embedding \eqref{equation_embedding}), acts on both $\mcal{O}_{\fa, \fb}$ and $\mcal{O}_{\fa', \fb'}$. 
The actions are transitive by Lemma \ref{lemma_transitive}. 
So, if there are elements $A \in \mcal{O}_{\fa^{\vphantom{\prime}}, \fb^{\vphantom{\prime}}}$ and $A^\prime \in \mcal{O}_{\fa^\prime, \fb^\prime}$ having the same stabilizer in $\rmso(m-1)$, then we obtain a unique $\rmso(m-1)$-equivariant 
diffeomorphism from $\mcal{O}_{\fa^{\vphantom{\prime}}, \fb^{\vphantom{\prime}}}$ to $\mcal{O}_{\fa^\prime, \fb^\prime}$ sending $A$ to $A^\prime$. 

To find such elements $A$ and $A'$, recall that the fiber  $\widetilde{\mcal{O}}_{\fa, \fb}$ of $\rho_{\fa, \fb} \colon \mcal{O}_{\fa, \fb} \to \mcal{O}_{\fb}^{(m-1)}$ at $I_{\fb}^{(m-1)}$ is the subset of 
$\mcal{O}_{\fa, \fb}$ consisting of elements of the form \eqref{equ_maxtrix_odd} or \eqref{equ_maxtrix_even} and 
an $\scr{L}_{\fb}$-homogeneous space  by Corollary \ref{cor_fiberhomegenous}. Here, $\scr{L}_{\fb} \subset \rmso(m-1)$ 
is the stabilizer of $I_{\fb}^{(m-1)}$ and is isomorphic to the product of $\rmso$-subgroups $\scr{L}_1, \cdots, \scr{L}_r$ of $\rmso(m-1)$ by Lemma \ref{lemma_levisubgroup}. 
Furthermore, Lemma \ref{lemma_levisubgroup}, \ref{lemma_case0}, \ref{lemma_case_odd}, and \ref{lemma_case_even}
imply that $\widetilde{\mcal{O}}_{\fa, \fb}$ is $\scr{L}_{\fb}$-equivariantly diffeomorphic to $\Pi(\widetilde{\mcal{O}}_{\fa', \fb'})$, which is a product of spheres of the following form$\colon$
\begin{equation}\label{equation_C}
	\begin{cases}
	\left\{ {\bf{z}} \in \R^{m-1} \mid \sum_{i=1}^{2\ell_1} z_i^2 = C_1, \cdots, \sum_{i=2\ell_{r-1}+1}^{m-2} z_i^2 = C_r, z_{m-1} = C_{r+1}  \right\} &\mbox{if $m = 2\ell + 2, b_\ell \neq 0$} \\
	\left\{ {\bf{z}} \in \R^{m-1} \mid \sum_{i=1}^{2\ell_1} z_i^2 = C_1, \cdots, \sum_{i=2\ell_{r-1}+1}^{m-1} z_i^2 = C_r   \right\} &\mbox{otherwise.}
	\end{cases}
\end{equation}
where $\Pi$ is the embedding of $\widetilde{\mcal{O}}_{\fa, \fb}$ into $\R^{m-1}$
given in \eqref{equ_projectionoadr} for some real numbers $C_1, \cdots, C_r \in \R^r_{\geq 0}$, $C_{r+1} \in \R_{\neq 0}$ uniquely determined by $(\fa, \fb)$.
Each $\scr{L}_i$-factor of $\scr{L}_{\fb}$ acts on $\Pi(\widetilde{\mcal{O}}_{\fa, \fb})$
as the standard linear action on the $i$-th sphere. 
Similarly for $(\fa^\prime, \fb^\prime)$, $\widetilde{\mcal{O}}_{\fa^\prime, \fb^\prime}$ is diffeomorphic to a product of spheres with radii $C_\bullet^\prime$ as in~\eqref{equation_C}. 
Note that  
$C_i = 0$ if and only if $C_i^\prime = 0$ as $(\fa^{\vphantom{\prime}}, \fb^{\vphantom{\prime}})$ and $(\fa^\prime, \fb^\prime)$ have the same inequality pattern (and hence $\scr{L}_{\fb} = \scr{L}_{\fb^\prime}$).

Let $A$ be the element of $\widetilde{\mcal{O}}_{\fa, \fb}$ whose image under $\Pi$ is 
\begin{equation}\label{equation_CCC}
{\bf z} := 
\begin{cases}
(\underbrace{\sqrt{C_1},0,\cdots,0}_{2\ell_1}, \underbrace{\sqrt{C_2},0,\cdots,0}_{2\ell_2-2\ell_1}, \cdots, \underbrace{\sqrt{C_r}, 0, \cdots 0}_{m-2 - 2\ell_{r-1}}, C_{r+1}) &\mbox{if $m = 2\ell + 2, b_\ell \neq 0$} \\
(\underbrace{\sqrt{C_1},0,\cdots,0}_{2\ell_1}, \underbrace{\sqrt{C_2},0,\cdots,0}_{2\ell_2-2\ell_1}, \cdots, \underbrace{\sqrt{C_r}, 0, \cdots 0}_{m-2 - 2\ell_{r-1}}) &\mbox{otherwise.}
\end{cases}
\end{equation}
Similarly, one can choose $A' \in \widetilde{\mcal{O}}_{\fa', \fb'}$. 
Clearly $A$ and $A'$ have the same stabilizer in $\scr{L}_{\fb^{\vphantom{\prime}}} = \scr{L}_{\fb'}$ (and hence in $\rmso(m-1)$). 
Define $\eta_{m-1}^{m}$ by letting $\eta_{m-1}^{m}(A) = A'$.
Moreover, $\eta^{m}_{m-1}$ can be generalized as follows. 

\begin{lemma}\label{lemma_fb}
	Let $\fa^{\vphantom{\prime}}_t := (1-t)\fa^{\vphantom{\prime}} + t \fa'$ and $\fb^{\vphantom{\prime}}_t := (1-t)\fb^{\vphantom{\prime}} + t \fb'$ for $t \in [0,1]$ so that $(\fa_0, \fb_0) = (\fa, \fb)$ and $(\fa_1, \fb_1) = (\fa', \fb')$. 
	Then there exists a parameter family of $\rmso(m-1)$-equivariant 
	diffeomorphisms
	\begin{equation}\label{equation_eta_pair}
		\eta^{m}_{m-1, t} \colon \mcal{O}_{\fa, \fb} \rightarrow \mcal{O}_{\fa_t, \fb_t}, \quad \quad \eta^{m}_{m-1,t}(A) = A_t	 
	\end{equation}
	where $A_t \in \widetilde{\mcal{O}}_{\fa_t, \fb_t}$ is taken as in \eqref{equation_CCC}.
\end{lemma}
 
 \begin{proof}
 	The proof follows from the fact that $(\fa_t, \fb_t)$'s have the same inequality pattern for every $t \in [0,1]$.
 \end{proof}
 
Setting  $\rho_m := \rho_{\frak{a}, \frak{b}}$ and $\rho^\prime_m := \rho_{\frak{a}^\prime, \frak{b}^\prime}$ in~\eqref{equ_projerhoab}, we have the following commutative diagram$\colon$ 
\begin{equation}\label{equation_commutativityaa}
	\xymatrix{
		  {\mcal{O}}_{\frak{a}^{\vphantom{\prime}}, \frak{b}^{\vphantom{\prime}}} \ar[d]_{\rho_m^{\vphantom{\prime}}} \ar[r]^{\eta^{m}_{m-1}}
                               & \mcal{O}_{\frak{a}^{{\prime}}, \frak{b}^{{\prime}}} \ar[d]^{\rho_{m}^{{\prime}}}  \\
 \mcal{O}_{\frak{b}^{\vphantom{\prime}}}^{(m-1)} \ar[r]_{\eta^{m-1}} & \mcal{O}_{\frak{b}^{{\prime}}}^{(m-1)}}
\end{equation}  
Also, we consider the following diagram:
\begin{equation}\label{equation_commutativityaa2}
	\xymatrix{
		  {\mcal{O}}_{\frak{a}, \frak{b}} \ar[d]_{\eta^{m}_{m-1}}^\simeq \ar[r]^{\zeta_{\fa, \fb}}
                               &  \mcal{O}^{(m)}_{\frak{a}} \ar[d]^{\eta^{m}}_\simeq  \\
 {\mcal{O}}_{\frak{a}^\prime, \frak{b}^\prime} \ar@{^{(}->}[r]_{\iota_{\fa', \fb'}}  & \mcal{O}_{\frak{a}^\prime}^{(m)}}
\end{equation} 
where $\iota_{\fa', \fb'}$ is the natural inclusion and $\zeta_{\fa, \fb} := (\eta^{m})^{-1} \circ \iota_{\fa', \fb'} \circ \eta^{m}_{m-1}$. Note that $\zeta_{\fa, \fb}$ is $\rmso(m-1)$-equivariant.

\begin{lemma}\label{lemma_isotopy}
	The map $\zeta_{\fa, \fb}$ is isotopic to the natural inclusion $\iota_{\fa, \fb} \colon \mcal{O}_{\frak{a}, \frak{b}} \to \mcal{O}^{(m)}_{\frak{a}}$.
\end{lemma}

\begin{proof}
	Notice that 
	\[
		\zeta_{\fa, \fb} = (\eta^{m})^{-1} \circ \iota_{\fa', \fb'} \circ \eta^{m}_{m-1} = (\eta^{m}_1)^{-1} \circ \iota_{\fa_1, \fb_1} \circ \eta^{m}_{m-1,1}.
	\]
	where $\fa_t, \fb_t, \eta^{m}_t, \eta^{m}_{m-1,t}$ are given in Lemma \ref{lemma_fa} and \ref{lemma_fb} for $t \in [0,1]$. Then the family of maps $\{ \zeta_{\fa, \fb}^t \mid t \in [0,1] \}$ where
	\[
		\zeta_{\fa, \fb}^t := (\eta^m_t)^{-1} \circ \iota_{\fa_t, \fb_t} \circ \eta^m_{m-1,t} : \mcal{O}_{\fa, \fb} ~(\to \mcal{O}_{\fa_t, \fb_t} \to \mcal{O}^{(m)}_{\fa_t} \to) ~\mcal{O}^{(m)}_{\fa}
	\]
	defines an isotopy from $\iota_{\fa, \fb}$ to $\zeta_{\fa, \fb}$ since $\zeta_{\fa, \fb}^0 = \iota_{\fa, \fb}$.
\end{proof}

Recall that the iterated bundle $E_n \to E_{n-1} \to \dots \to E_3 \to E_2$ in the diagram~\eqref{figure_iterated_bundle} is constructed as the pull-backed bundles via the inclusions 
$\iota_{\fa, \fb} \colon {\mcal{O}}_{\frak{a}, \frak{b}} \to \mcal{O}_{\frak{a}}^{(m)}$'s. 
If we replace  $\iota_{\fa,\fb}$'s in~\eqref{figure_iterated_bundle} with $\zeta_{\fa, \fb}$'s in~\eqref{equation_commutativityaa2}, then we obtain a new iterated bundle
\begin{equation}\label{equ_iteratedbundletwo}
F_n \to F_{n-1} \to \dots \to F_3 \to F_2 = E_2.
\end{equation}
Then the following lemma follows from Lemma~\ref{lemma_isotopy}.

\begin{lemma}\label{lemma_doesnotaffectop}
Two total spaces $E_n$ in~\eqref{figure_iterated_bundle} and $F_n$ in~\eqref{equ_iteratedbundletwo} are diffeomorphic.
\end{lemma}

Now, we are ready to prove Proposition~\ref{proposition_twofibersarediffeo}.

\begin{proof}[Proof of Proposition~\ref{proposition_twofibersarediffeo}]
We start from two squared diagrams~\eqref{equation_commutativityaa} and~\eqref{equation_commutativityaa2} as in~\eqref{equation_commutativity33}.
Proceeding inductively, we obtain two parallel diagrams consisting of sequences $(\rho^\prime_\bullet, \iota)$ and $(\rho_\bullet, \zeta)$ of bundles and embeddings. 
\begin{equation}\label{equation_commutativity33}
	\xymatrix{
		 {} & \mcal{O}_{\frak{a}, \frak{b}} \ar[dd]^{\rho_m} \ar[rr]^{\zeta_{\fa, \fb}} & & \mcal{O}_{\frak{a}}^{(m)}  \\
 {\mcal{O}}_{\frak{a}^\prime, \frak{b}^\prime} \ar[dd]_{\rho^\prime_m} \ar[ru]^{\eta^m_{m-1}}  \ar@{^{(}->}[rr]_{\iota_{\fa', \fb'}} & & \mcal{O}_{\frak{a}^\prime}^{(m)} \ar[ru]_{\eta^m} & \\
{} &  {\mcal{O}_{\frak{b}}^{(m-1)}}  & & \\
 {\mcal{O}_{\frak{b}^\prime}^{(m-1)}} \ar[ru]_{\eta^{m-1}} & & & 
 }
\end{equation} 
By pulling back the bundles to the left, we obtain two iterated bundles
$$
\begin{cases}
F_n \to F_{n-1} \to \dots \to F_3 \to F_2, \\
E_n^\prime \to E_{n-1}^\prime \to \dots \to E_3^\prime \to E_2^\prime
\end{cases}
$$ 
which are obtained from $(\rho^\prime_\bullet, \iota)$ and $(\rho_\bullet, \zeta)$ respectively. 
Moreover, the diffeomorphisms $\eta^m_{m-1}$'s and $\eta^m$'s induce a bundle isomorphism between them. 
In particular, $F_n$ and $E_n^\prime$ are diffeomorphic.
Notice that $E_n^\prime$ is diffeomorphic to $(\Phi^{(n)}_\lambda)^{-1}({\bf{u}}^\prime)$.
By Lemma~\ref{lemma_doesnotaffectop}, $F_n$ is diffeomorphic to $E_n \simeq (\Phi^{(n)}_\lambda)^{-1}({\bf{u}})$. In sum,
$$
(\Phi_\lambda^{(n)})^{-1} ({\bf{u}^\prime}) \simeq E^\prime_n \simeq F_n \simeq E_n \simeq (\Phi_\lambda^{(n)})^{-1} ({\bf{u}}).
$$
This completes the proof of Proposition~\ref{proposition_twofibersarediffeo}. 
\end{proof}

Finally, we are ready to complete the proof of Theorem~\ref{thm_fiber}.

\begin{proof}[Proof of Theorem~\ref{thm_fiber}]
Proposition~\ref{proposition_mainiterated},~\ref{prop_isotropic}, and~\ref{proposition_twofibersarediffeo} yield Theorem~\ref{thm_fiber}
\end{proof}

\vspace{0.1cm}
\section{Classification of Lagrangian faces}
\label{secClassificationOfLagrangianFaces}

Every GC fiber is an isotropic submanifold of $\mcal{O}_\lambda^{(n)}$ by Theorem~\ref{thm_fiber} so that
a GC fiber is Lagrangian if and only if its dimension equals the half of the real dimension of $\mcal{O}_\lambda^{(n)}$.
In this section, we enhance the combinatorial process in Section~\ref{secTheTopologyOfGelfandCetlinFibers} to derive a simple way for computing the dimension of the fiber. As a byproduct, we classify the positions of Lagrangian fibers of the GC systems of the $\mathrm{SO}(n)$-orbits.

For this purpose, we need new types of blocks.

\begin{definition} Let $k \in \N$ be a positive integer.
\begin{enumerate}
\item (Definition 5.16. in \cite{CKO}) For a lattice point $(a_0,b_0) \in \Z^2$, the $L_k$\emph{-block} at $(a_0, b_0)$, denoted by $L_k(a_0,b_0)$, is defined by 
\begin{equation*}\label{eq_lkunion}
L_k(a_0, b_0) := \bigcup_{(a_0, b_0)+ (a,b)} \blacksquare^{(a,b)}
\end{equation*}
where the union is taken over all $(a,b) \in \mathbb{N}^2$ satisfying either
$(a,b) = (i, 0)$ or $(a,b) = (0, i)$
for $i  =0,\cdots, k-1$. We simply call $L_k(a_0,b_0)$ for some $k \in \N$ and $(a_0, b_0) \in \Z^2$ an $L$-block.

\item For a lattice point $(a_0,b_0) \in \Z^2$, the $I_k$\emph{-block} at $(a_0, b_0)$, denoted by $I_k(a_0, b_0)$, is defined by
\begin{equation*}\label{eq_lkunion}
	I_k(a_0, b_0) := \bigcup_{(a_0, b_0) + (a,b)} \blacksquare^{(a,b)}
\end{equation*}
where the union is taken over all $(a,b) \in \mathbb{N}^2$ such that $(a,b) = (0, i)$ for $i = 0, \cdots, k-1$. We simply call $I_k(a_0,b_0)$ for some $k \in \N$ and $(a_0, b_0) \in \Z^2$ an $I$-block.
\end{enumerate}
\end{definition}


\begin{figure}[h]
	\vspace{0.1cm}
	\scalebox{0.8}{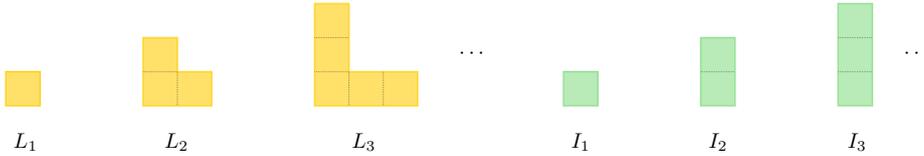}
		\vspace{-0.2cm}
	\caption{\label{Fig_LIblocks} $L$-blocks and $I$-blocks.}	
\end{figure}
\vspace{-0.1cm}


\begin{definition}\label{def_LIfillable}
We say that a face $(\gamma, \delta(\gamma))$ is \emph{fillable by} $L$-blocks and $I$-blocks if the ladder diagram $\Gamma_\lambda^{(n)}$ can be covered by $L$-blocks and $I$-blocks 
obeying the following rules:
\begin{enumerate}
\item the interiors of any two blocks must not intersect each other. 
\item the interiors of the $L$-blocks must not intersect the isogram $\gamma$, 
\item the rightmost and the top edge of each $L$-block must be contained in the isogram $\gamma$,  and 
\item the top edge of each $I$-block must be contained in the base of $\gamma$.
\end{enumerate} 
We call a face {\em Lagrangian} if it is fillable by $L$-blocks and $I$-blocks. 
\end{definition}

\begin{remark}
The dimension of a GC fiber only depends on $\gamma$, \emph{not} on the coastline $\delta$ by Proposition~\ref{proposition_twofibersarediffeo}.
\end{remark}

\begin{proposition}\label{prop_Lagrangianfacetests}
Let $(\gamma, \delta(\gamma))$ be the face of $\Gamma_\lambda^{(n)}$ corresponding to a face $f$. Then the following are equivalent.
\begin{enumerate}
\item The face $(\gamma, \delta(\gamma))$ is Lagrangian.
\item For a relative interior point ${\bf{u}}$ of $f$, the fiber $\left(\Phi_\lambda^{(n)}\right)^{-1}({\bf{u}})$ is a Lagrangian submanifold of $\mcal{O}^{(n)}_\lambda$.
\item For any relative interior point ${\bf{u}}$ of $f$, the fiber $\left(\Phi_\lambda^{(n)} \right)^{-1}({\bf{u}})$ is a Lagrangian submanifold of $\mcal{O}^{(n)}_\lambda$.
\end{enumerate}
\end{proposition} 

\begin{remark}
The authors with Oh characterize all faces which have Lagrangian fibers at their interior for the co-adjoint $\mathrm{U}(n)$-orbits by filling the diagram by $L$-blocks, see \cite[Corollary 5.23]{CKO}. Unlike the $\mathrm{U}(n)$-case, one should pay attention to the base of the isogram $\gamma$ in a face. Even though the GC patterns for two points in the polytope are same, whether a component of the GC system becomes zero or not \emph{does} make a distinction on the dimensions of GC fibers, see Example~\ref{example_Lagrangianfaceclass}. 
\end{remark}

\begin{proof}[Proof of Proposition~\ref{prop_Lagrangianfacetests}]
By Proposition~\ref{proposition_twofibersarediffeo}, the fiber over ${\bf{u}}_1$ is diffeomorphic to that over ${\bf{u}}_2$ if two points are contained in the relative interior of $f$ so that the second and the third statement are equivalent.
It suffices to prove the equivalence of the first and the second. 

Let us cover an isogram $\gamma$ by $L$-blocks and $I$-blocks
\begin{equation}\label{equ_coveringbloIL}
	\{L_{k_i}(a_i, b_i) ~|~ i=1,\cdots,r \} \, \cup \, \{ I_{\ell_i}(c_i, d_i) ~|~ j=1,\cdots,s \},
\end{equation}
satisfying the conditions $(1),(2),(3)$ and $(4)$ in Definition \ref{def_LIfillable}. Note that $I$-blocks (resp. $L$-blocks) \emph{cannot} be located at the above (resp. below) base of $\gamma$ because of $(4)$ (resp. $(3)$). There does not exist any ambiguity putting $L$-blocks. To obtain the maximum filling, we put the $I$-blocks so that the bottom edge of each $I$-block must lie on the bottom of a diagonal box. 

Observe that
\begin{itemize}
	\item The block $L_{k_i}(a_i, b_i)$ has area $2k_i - 1$, 
	\item The assumption that one can put $L_{k_i}(a_i, b_i)$-block into $\gamma$ yields that 
	\[
		  {u}_{a_i, b_i + k_i } > {u}_{a_i , b_i + k_i - 1}  = \cdots = {u}_{a_i, b_i} =  \cdots = {u}_{a_i + k_i - 1, b_i}  > \max( {u}_{a_i + k_i, b_i }, 0 ).
	\]	
	Because of the min-max principle, there exists a unique $M_{k_i}$-block satisfying 
		\begin{enumerate}
			\item its leftmost unit box is located at $(a_i, b_i + k_i  - 1)$
			\item it does \emph{not} intersect the isogram $\gamma$ in its interior, see Figure \ref{Fig_NMIL}. 	
		\end{enumerate}
	\item Each $L_{k_i}(a_i, b_i)$ corresponds to a subsequence of type $(2), (4),$ or $(8)$ in \eqref{equ_pattern_odd} and \eqref{equ_pattern_even}. Therefore, by Lemma \ref{lemma_descriptionofsab}, each $L_{k_i}(a_i, b_i)$ detects $\mathbb{S}^{2k_i - 1}$-factor in~\eqref{equ_fibersfromblocks} at the $(a_i + b_i + k_i)$-th stage.
\end{itemize}

Observe that
\begin{itemize}
	\item The block $I_{\ell_i}(c_i, d_i)$ has area $\ell_i$, 
	\item The assumption that one can put $I_{\ell_i}(c_i, d_i)$-block into $\gamma$ yields that 
	\[
		  {u}_{c_i, d_i + \ell_i } > {u}_{c_i , d_i + \ell_i - 1}  = \cdots = {u}_{c_i, d_i} = 0.
	\]	
	Because of the min-max principle, there exists a unique 
		$$
			\begin{cases}
				M_{\ell_i}\textup{- block} \quad \mbox{if $\ell_i$ is odd} \\
				N_{\ell_i}\textup{- block} \quad \mbox{if $\ell_i$ is even}
			\end{cases}	
		$$
	satisfying 
		\begin{enumerate}
			\item its leftmost unit box is located at $(c_i, d_i + \ell_i  - 1)$
			\item it does \emph{not} intersect the isogram $\gamma$ in its interior, see Figure \ref{Fig_NMIL}. 	
		\end{enumerate}
	\item Each $I_{\ell_i}(c_i, d_i)$ corresponds to the subsequence of type $(1)$ or $(5)$ in \eqref{equ_pattern_odd} and \eqref{equ_pattern_even}. Again by Lemma \ref{lemma_descriptionofsab}, each 
$I_{\ell_i}(c_i, d_i)$ detects $\mathbb{S}^{\ell_i}$-factor in~\eqref{equ_fibersfromblocks} at the $(c_i + d_i + \ell_i )$-th stage. 
\end{itemize}

\begin{figure}[h]
	\scalebox{0.9}{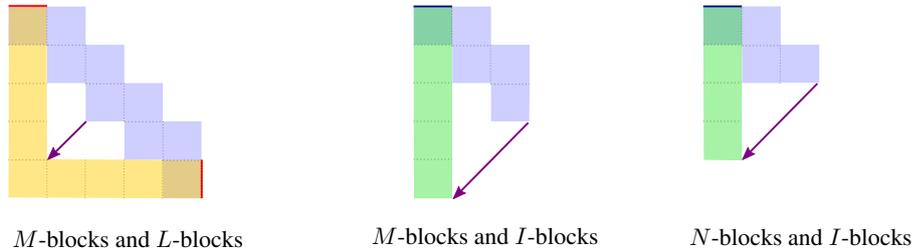}
	\caption{\label{Fig_NMIL} The correspondence between $M$-and $N$-blocks and $I$-and $L$-blocks.}	
\end{figure}
\vspace{-0.1cm}

Summing up, we conclude that 
\begin{equation}\label{equ_dimension}
	\dim \left(\Phi_\lambda^{(n)} \right)^{-1}({\bf{u}}) = \sum_{k=3}^{n} \dim S_k(\gamma) = 	
	\sum_{i=1}^r \text{area $\left(L_{k_i}(a_i, b_i) \right)$} + \sum_{i=1}^s \text{area $\left(I_{\ell_i}(c_i, d_i)\right)$}. 
\end{equation}

Now we go back to the proof of Proposition \ref{prop_Lagrangianfacetests}.
Since $(\gamma,\delta)$ is fillable by $L$ and $I$-blocks if and only if 
\[
	\frac{1}{2} \dim_\R \mcal{O}_\lambda^{(n)} = \text{area $\left(\Gamma_\lambda^{(n)}\right)$} = \sum_{i=1}^r \text{area $\left(L_{k_i}(a_i, b_i) \right)$} + \sum_{i=1}^s \text{area $\left(I_{\ell_i}(c_i, d_i)\right)$} \underbrace{=}_{\eqref{equ_dimension}} 
	\dim \left(\Phi_\lambda^{(n)}\right)^{-1}({\bf{u}}).
\]
It completes the proof.
\end{proof}

As a by-product, we obtain the dimension formula for the fiber from~\eqref{equ_dimension}.

\begin{corollary}\label{cor_fillable}
Let $f$ be the face containing a given point $\bf{{u}}$ in its relative interior. Fill the corresponding isogram $\gamma_f$ by $L$-blocks and $I$-blocks as many as possible obeying the conditions in Definition~\ref{def_LIfillable}.
Then the dimension of the fiber over $\bf{{u}}$ is exactly the area of the regions filled by $L$-blocks and $I$-blocks. 
\end{corollary}

\begin{example}\label{example_Lagrangianfaceclass}
Let  $n = 5$ and $\lambda = (3, 0)$. The orbit $\mcal{O}^{(5)}_\lambda$ is the orthogonal Grassmannian $\mathrm{OG}(1, \C^5)$. 
Then the GC polytope $\Delta_\lambda^{(5)}$ is defined by the intersection of the four half-planes$\colon$ 
\[
3 - u_{1,3} \geq 0, \, u_{1,3} - u_{1,2} \geq 0, \, u_{1,2} - u_{1,1} \geq 0, \, u_{1,2} + u_{1,1} \geq 0, 
\]
which is a simplex as in Figure~\ref{Fig_OG15}. In this example, there are a zero-dimensional Lagrangian face, one two-dimensional Lagrangian face, and one three-dimensional Lagrangian face which are listed in Figure~\ref{Fig_LagOG15}. 

Consider two points $(0,0,0)$ and $(1,1,1)$ satisfying the same GC patterns. ($0 = 0 = |0|$ and $1 = 1 = |1|$) However, the dimensions of two fibers $(\Phi_\lambda^{(5)})^{-1}(0,0,0) \simeq \mathbb{S}^3$ and $(\Phi_\lambda^{(5)})^{-1}(1,1,1) \simeq \mathbb{S}^1$ are different. As depicted in the first and last graphs in Figure~\ref{Fig_LagOG15}, one may put an $I_3$-block for the first case but cannot put an $I_3$-block for the second case.

\begin{figure}[h]
	\scalebox{0.7}{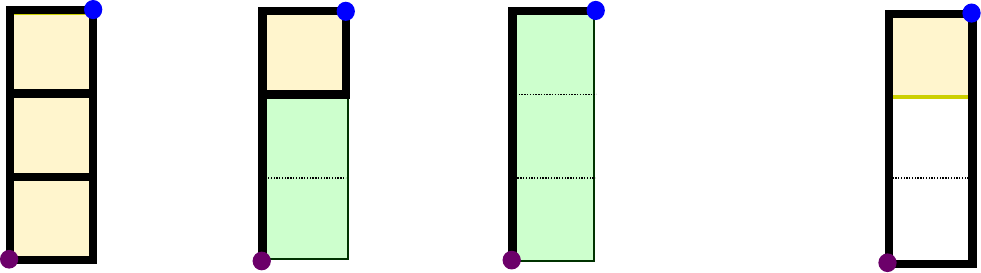}
	\vspace{-0.1cm}
	\caption{\label{Fig_LagOG15} The Lagrangian faces and a non-Lagrangian face of $\Delta_\lambda^{(5)}$}	
\end{figure}
\end{example}
\vspace{-0.2cm}

\begin{example}
Let $n = 6$ and $\lambda = (3, 3, 3)$. Then $\mcal{O}^{(6)}_\lambda$ is the orthogonal Grassmannian $\mathrm{OG}(3, \C^6)$. By classifying the faces fillable by $L$-blocks and $I$-blocks, one can classify all Lagrangian faces. In this case, it has six faces such that any fiber located at their interior is Lagrangian as depicted in Figure~\ref{Fig_LagOG36}. 

\begin{figure}[h]
	\scalebox{0.7}{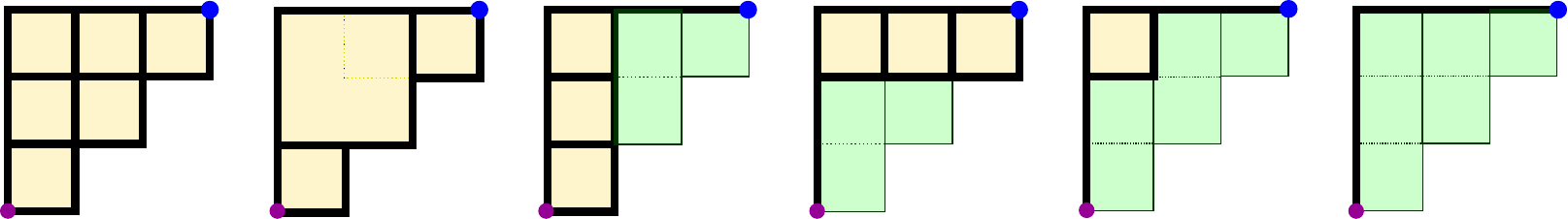}
		\vspace{-0.1cm}
	\caption{\label{Fig_LagOG36} The Lagrangians faces of $\Delta_\lambda^{(6)}$}	
\end{figure}
\end{example}
\vspace{-0.1cm}

\begin{remark}
With the dimension formula in~Corollary~\ref{cor_fillable} in hand, the dimensions of fibers can be quickly calculated. Because of the presence of $I_2$-block, a non-toric fiber can appear at a point in a face of codimension $2$. 
In the type of $\mathrm{U}(n)$, every fiber at the relative interior of any strata of codimension $\leq 2$ is an isotropic torus. Meanwhile, for the case of the $\mathrm{SO}(n)$-orbits, it is only guaranteed that every fiber at the relative interior of any strata of codimension $\leq 1$ is an isotropic torus.

Combining Theorem~\ref{thm_diacorrespo} and Corollary~\ref{cor_fillable}, one can also calculate the dimension of the inverse image of a face. 
In Example~\ref{example_Lagrangianfaceclass}, it contains the edge $\{ (u_{1,1}, u_{1,2}, u_{1,3}) ~|~ u_{1,1} = u_{2,2} = 0, 0 \leq u_{1,3} \leq 3 \}$ 
such that the inverse image of its interior is of codimension two. By applying the formula for the Maslov index in \cite{CK_mono}, one can see that there exists a holomorphic disc of Maslov index two intersecting the inverse image over the edge. 
This feature implies that the Floer theoretical disc potential \cite{FOOO} might contain more than the terms corresponding to the facets. It contrasts with the case of partial flag manifolds of type $A$ where the Floer theoretical disc potential consists of the discs corresponding to the facets, see Nishinou--Nohara--Ueda \cite{NNU1}. 
\end{remark}

\end{document}